\documentclass[12pt,reqno]{amsart}
\usepackage[margin=1in]{geometry}
\usepackage{amsmath,amssymb,amsthm,graphicx,amsxtra, setspace}
\usepackage[utf8]{inputenc}
\usepackage{mathrsfs}
\usepackage{hyperref}
\usepackage{upgreek}
\usepackage{mathtools}
\usepackage{xcolor}
\usepackage[mathcal]{euscript}
\allowdisplaybreaks

\usepackage{subcaption} 

\usepackage[pagewise]{lineno}

\DeclareMathAlphabet{\mathpzc}{OT1}{pzc}{m}{it}

\usepackage[cyr]{aeguill}

\colorlet{darkblue}{blue!50!black}

\hypersetup{
	colorlinks,%
	citecolor=blue,%
	filecolor=red,%
	linkcolor=red,%
	urlcolor=blue,%
	pdfnewwindow=true,%
	pdfstartview={FitH}
}

\newtheorem{theorem}{Theorem}[section]
\newtheorem{lemma}[theorem]{Lemma}
\newtheorem{proposition}[theorem]{Proposition}

\newtheorem{corollary}[theorem]{Corollary}
\newtheorem{definition}[theorem]{Definition}
\newtheorem{problem}[theorem]{Problem}

\newtheorem{remark}[theorem]{Remark}
\newtheorem{algorithm}[theorem]{Algorithm}

\newtheorem{hypothesis}[theorem]{Hypothesis}

\allowdisplaybreaks

\let\originalleft\left
\let\originalright\right
\renewcommand{\left}{\mathopen{}\mathclose\bgroup\originalleft}
\renewcommand{\right}{\aftergroup\egroup\originalright}

\renewcommand{\d}{\/\mathrm{d}\/}

\def\L{\mathbb{L}}

\def\I{\mathrm{I}}

\def\f{\boldsymbol{f}}

\def\y{\boldsymbol{y}}

\def\X{\mathbb{X}}
\def\x{\boldsymbol{x}}

\def\g{\boldsymbol{g}}

\def\z{\boldsymbol{z}}
\def\bv{\boldsymbol{v}}
\def\V{\mathcal{v}}
\def\bw{\boldsymbol{w}}

\def\N{\mathbb{N}}

\def\V{\mathcal{V}}

\def\bu{\boldsymbol{u}}
\def\H{\mathbb{H}}
\def\n{\boldsymbol{n}}

\newcommand{\R}{\mathbb{R}}

\renewcommand{\d}{\/\mathrm{d}\/}

\newcommand{\Addresses}{{
		\footnote{
			
			\noindent \textsuperscript{1,2}Department of Mathematics, Indian Institute of Technology Roorkee-IIT Roorkee,
			Haridwar Highway, Roorkee, Uttarakhand 247667, INDIA.\par\nopagebreak
			\noindent  \textit{e-mail:} \texttt{Manil T. Mohan: maniltmohan@ma.iitr.ac.in, maniltmohan@gmail.com.}
			
			\textit{e-mail:} \texttt{Wasim Akram: wakram2k11@gmail.com.}
			
			\noindent \textsuperscript{*}Corresponding author.
			
			\textit{Key words:} Stationary Convective Brinkman-Forchheimer Extended Darcy equations, Optimal Control, Hemivariational Inequality, finite element discretization, convergence analysis. 
			
			Mathematics Subject Classification (2020): Primary: 65K10; Secondary: 35Q35, 49J20, 76D03.

}}}

\begin{document}
	
	\title[Optimal control of hemivariational inequality for 2D and 3D CBFeD equations]{Optimal Control  of a Hemivariational Inequality of Stationary Convective Brinkman-Forchheimer Extended Darcy equations with Numerical Approximation 
		\Addresses}
	\author[W. Akram and M. T. Mohan]
	{Wasim Akram\textsuperscript{2} and Manil T. Mohan\textsuperscript{1*}}
	
	\maketitle

	\begin{abstract}
This paper studies an optimal control problem for a stationary convective Brinkman-Forchheimer extended Darcy (CBFeD) hemivariational inequality in two and three dimensions, subject to control constraints, and develops its numerical approximation. The hemivariational inequality provides the weak formulation of a stationary incompressible fluid flow through a porous medium, governed by the CBFeD equations, which account for convection, damping, and nonlinear resistance effects. The problem incorporates a non-leak boundary condition and a subdifferential friction-type condition. We first analyze the stability of solutions with respect to perturbations in the external force density and the superpotential. Next, we prove the existence of a solution to the optimal control problem, where the external force density acts as the control variable. We then propose a numerical scheme for solving the optimal control problem and establish its convergence. For concreteness, the numerical method is implemented using finite element discretization. Finally, we provide some numerical examples to validate the theory developed. 
	\end{abstract}

	\section{Introduction}\label{sec1}\setcounter{equation}{0}
The study of optimal control problems governed by hemivariational inequalities (HVIs) associated with partial differential equations, particularly the Stokes and Navier-Stokes systems, has become increasingly important due to its capacity to model complex fluid behavior characterized by nonsmooth and nonmonotone dynamics. Such problems naturally arise in areas like fluid mechanics, contact problems, and processes involving multivalued or nonmonotone constitutive laws. In the context of the Stokes and Navier-Stokes equations, HVIs are often used to describe phenomena such as slip conditions at boundaries or irregular frictional interactions. These models, when integrated with fluid dynamics, result in highly nonlinear and mathematically demanding systems. Despite their complexity, they provide a robust framework for addressing optimal control scenarios in which the goal is to determine an appropriate control input, such as an applied force or boundary condition, that minimizes a predefined cost functional. This optimization must be achieved while simultaneously satisfying the HVI-constrained fluid equations, where the state variables (typically velocity and pressure) are subject to nonsmooth and possibly multivalued relations. Core analytical challenges include establishing the existence of solutions, deriving optimality conditions under nonsmooth settings, and designing numerical algorithms that are both accurate and efficient. The practical relevance of such problems is evident in engineering contexts such as lubricated surface flows, interactions between fluids and deformable structures, and transport through porous media with irregular or reactive boundaries.

\subsection{Literature review}
The optimal control of stationary Stokes hemivariational inequalities and their numerical approximations has been extensively studied in the literature; see, for instance, \cite{XCRGWH,CFWH1,CFWH2}. Building on these developments, the authors in \cite{WWXCWH} extended the analysis to the stationary Navier-Stokes setting, focusing on both the control problem and its numerical approximation. Their model addresses the steady flow of an incompressible, viscous fluid, incorporating a nonleak boundary condition along with a friction-type constraint expressed through a subdifferential formulation. Further progress was made in \cite{OCRNM}, where the optimal control of nonstationary Navier-Stokes equations was studied under nonlinear boundary conditions described by the Clarke subdifferential, with the objective of minimizing a general cost functional dependent on the control through the system’s state.

Hemivariational inequalities, first introduced by Panagiotopoulos (\cite{PDP}), offer a powerful mathematical framework for modeling systems characterized by nonsmooth and nonmonotone interactions, which frequently arise in mechanics and engineering applications. They are particularly useful in contact mechanics with friction, where traditional variational methods fail to capture nonconvex behavior, as well as in fluid dynamics, material science, and phase transition problems. These inequalities also arise in optimal control, structural optimization, and systems with hysteresis or unilateral constraints, making them a versatile tool in both theoretical analysis and engineering applications. Over the years, a robust theoretical foundation for these inequalities has been established, as documented in key references such as \cite{ZNPDP}. Due to the inherent complexity and lack of closed-form solutions, hemivariational inequalities are typically addressed through numerical methods. Foundational contributions to numerical techniques and algorithmic approaches can be found in \cite{JHMMPD}, while a comprehensive overview of more recent advancements and theoretical insights into the numerical treatment of hemivariational inequalities is provided in \cite{WHMS}.

In the context of the Navier-Stokes equations, when the boundary conditions involve nonsmooth but monotone relations, the associated weak formulation leads to a variational inequality, a framework that has been explored in several studies, including \cite{AYC,HFu}. However, when the boundary behavior is characterized by nonsmooth and nonmonotone relations, the weak formulation naturally evolves into a hemivariational inequality. The foundational work on the well-posedness of such Navier-Stokes hemivariational inequalities was carried out in \cite{SMi2,SMAO}, where the authors utilized an abstract surjectivity theorem for pseudomonotone operators to establish existence results. The corresponding optimal control problem and the existence of optimal solutions were later introduced in \cite{SMi1}. In contrast to the operator-theoretic approach used in \cite{SMi2,SMAO}, the studies in \cite{WHFJYY,MLWH1} demonstrated well-posedness by relying solely on fundamental tools from functional analysis. 
	
The body of research on hemivariational inequalities (HVIs) associated with the convective Brinkma-Forchheimer (CBF) and its extended Darcy (CBFeD) models remains relatively sparse. The paper \cite{WHHQLM}  investigated a Stokes hemivariational inequality for incompressible fluid flow with damping, established well-posedness via a minimization approach, and developed mixed finite element methods with corresponding error estimates. The work in \cite{JJSGMTM} investigates boundary HVIs for both stationary and time-dependent 2D and 3D CBF equations, incorporating a no-slip condition along with a Clarke subdifferential relation linking pressure and the normal component of velocity. For the stationary case, the existence and uniqueness of weak solutions are established using a surjectivity theorem for pseudomonotone operators. In \cite{WAMTM}, a mixed finite element approach is developed for the stationary CBFeD HVI, with well-posedness proven using techniques analogous to those applied for the CBF case, followed by a comprehensive error analysis of the numerical scheme. The paper \cite{WWXLCWH} establishes well-posedness, proposes a mixed finite element method with optimal error estimates, and validates an efficient iterative algorithm for the stationary Navier-Stokes hemivariational inequality with nonlinear damping and frictional slip boundary conditions.

The control of fluid flow remains a topic of significant interest across scientific and engineering disciplines, driven by its critical role in a wide range of complex technological systems. While optimal control problems governed by the classical Navier-Stokes equations and standard boundary value problems have been extensively explored in the literature (see, e.g., \cite{FART,AAFFEO,ECMMJPR,MDGLH,MDG,OpVf}, etc. and references therein), recent attention has shifted toward more challenging scenarios involving nonsmooth and nonmonotone boundary conditions. This has led to the development of optimal control frameworks based on variational and hemivariational inequalities, as discussed in works such as \cite{CFWH1,SMi,YBXMS,WWXCWH}. In this work \cite{WWXCWH},  focused on the optimal control of the Navier-Stokes hemivariational inequality, where the control variable is the external force density. The analysis included a finite element approximation of the problem and established convergence of the proposed method. 
We build upon the analysis presented in \cite{WWXCWH} by advancing it to the next level, focusing on the optimal control of a hemivariational inequality governed by the stationary CBFeD equations and formulating its numerical approximation.
	
\subsection{Contributions and novelties}
The present paper investigates an optimal control problem governed by a stationary convective Brinkman-Forchheimer extended Darcy (CBFeD) hemivariational inequality (HVI) in both two and three spatial dimensions. The control problem is formulated under the presence of control constraints, and considerable attention is devoted to the development and analysis of a corresponding numerical approximation scheme.
The hemivariational inequality serves as the weak formulation of a steady-state model describing the flow of an incompressible viscous fluid through a porous medium. The governing equations are the CBFeD equations, which generalize Darcy’s law by incorporating additional physical effects such as inertia (convection), viscous damping (Brinkman term), and nonlinear drag or resistance (Forchheimer term). These extended terms allow the model to better capture complex behaviors of fluid motion in porous environments, particularly under moderate to high flow rates.
The boundary conditions for the problem include a non-leak condition, which prevents fluid from escaping through parts of the domain boundary, and a friction-type condition expressed using the Clarke subdifferential, which introduces a nonmonotone, possibly multivalued relation modeling physical phenomena like nonlinear slip or dry friction.
	
As a preliminary step toward analyzing the optimal control problem, we begin by examining the stability of solutions to the CBFeD hemivariational inequality with respect to variations in the external force density and the superpotential. This stability analysis is not only essential for control but is also of independent theoretical interest. We note that related stability results for elliptic HVIs can be found in \cite{WHYL,YBXMS} and for stationary Navier-Stokes HVI, one can refer to  \cite{WWXCWH}. 
Subsequently, the paper establishes the existence of an optimal control that minimizes a given cost functional while ensuring that the state variables (the fluid velocity field and the pressure) satisfies the CBFeD hemivariational inequality. The control variable is taken to be the external force density, allowing for the steering of the fluid flow by adjusting forces acting within the domain.
To solve the problem computationally, the paper introduces a numerical scheme based on the finite element method (FEM). The discrete version of the problem is rigorously analyzed, and we prove the convergence of the numerical approximation under suitable assumptions. This ensures that as the discretization becomes finer, the computed solutions approximate the continuous solution with increasing accuracy. Furthermore, alternative numerical strategies, including the discontinuous Galerkin method (\cite{FWMLWH}) and the virtual element method (\cite{MLFWWH}), offer promising directions for future research and can also be adapted to handle the nonsmooth structure inherent in such control problems.
Finally, to demonstrate the practical effectiveness of the proposed approach, the paper presents numerical experiments. These examples illustrate the theoretical results and provide insights into the behavior of optimal solutions under various physical and geometric configurations.
		
Compared to the Navier-Stokes HVI, the CBFeD HVI with fast-growing nonlinearities, specifically when the absorption exponent $r\in(3,\infty)$, offers certain advantages, particularly regarding the uniqueness of solutions. In this case, Theorem \ref{thm-unique} guarantees uniqueness under conditions \eqref{eqn-unique-con-1} or \eqref{eqn-unique-con-11}, which are independent of the external force $\f$. As a result, for the corresponding optimal control problem, no smallness assumption on the control variable $\f$ is needed. In contrast, when $r\in[1,3]$, condition \eqref{eqn-unique-con-2} indicates that a smallness requirement on $\f$ is essential to ensure uniqueness, thereby imposing restrictions on the admissible controls. Owing to the application of Sobolev embedding results, our analysis is confined to  $r\in[1,\infty)$ for the case $d=2$ and $r\in[1,5]$ when  $d=3$. To the best of our knowledge, this work is the first to address the optimal control of a hemivariational inequality governed by the stationary CBFeD equations, along with the development and analysis of its  finite element  numerical approximation.

Finally, in Section~\ref{sec7}, we provide three numerical examples that validate our theoretical findings. In addition, we propose Algorithm~\ref{algorithm}, an efficient iterative method for solving the mixed FEM-based optimal control hemivariational inequality, which yields desirable convergence properties and computational accuracy. To the best of our knowledge, these are the first computations of their kind in the literature. We consider a range of parameter settings, including cases for the Navier-Stokes equations (with and without damping) as well as for the full CBFeD system. In existing works (\cite{WAMTM, WHKCFJ, WWXLCWH}), the HVI is typically applied in the absence of control, or when the control is prescribed. In contrast, our approach seeks to determine an appropriate control, beginning with an initial guess, we solve the HVI to obtain the corresponding state, and then iteratively update the control via a subgradient scheme (see Section~\ref{sec7} for details). The simulations presented in Section~\ref{sec7} confirm the theoretical results.

\subsection{Structure of the paper} 
The remainder of the paper is organized as follows. Section \ref{sec2} introduces the necessary notation and foundational concepts that will be used throughout the work. The existence and uniqueness problem for a stationary CBFeD hemivariational inequality is addressed in Section \ref{sec3} (Theorems \ref{thm-main-hemi} \ref{thm-unique} and \ref{thm-equivalent}).  In Section \ref{sec4}, we analyze the stability of solutions to the stationary CBFeD hemivariational inequality (Problems \ref{prob-hemi-1} and \ref{prob-hemi-2}) under perturbations in the external force density and the superpotential (Theorem \ref{thm-conv} and Corollary \ref{cor-conver}).  In Section \ref{sec5}, we consider an associated optimal control problem, incorporating control constraints, and show the existence of an optimal control (Theorem \ref{thm-optimal-exis}). Section \ref{sec6} is dedicated to the numerical approximation of this control problem, where we establish the convergence of the proposed numerical scheme within a general framework (Theorems \ref{thm-discrete-conv} and \ref{thm-conv-cost}. Finally, in Section \ref{sec7}, we provide some numerical examples to validate the theory developed in Section \ref{sec6}. 
	
\section{Mathematical Formulation}\label{sec2}\setcounter{equation}{0}
The main objective of this section is to present the fundamental mathematical preliminaries required for the analysis in this work. Additionally, we set up the functional framework relevant to the problem under investigation.

\subsection{Preliminaries}
Throughout the paper, we restrict ourselves to function spaces defined over $\mathbb{R}$. Let $\mathbb{X}$ denote a normed space with norm $\|\cdot\|_{\mathbb{X}}$. We use $\mathbb{X}^*$ to represent its topological dual and $\langle \cdot, \cdot \rangle$ for the action of $\mathbb{X}^*$ on $\mathbb{X}$. Unless explicitly stated, $\mathbb{X}$ will be treated as a Banach space.

We start with the definition of a locally Lipschitz function.
\begin{definition}
	We say that a function $f:\X\to\mathbb{R}$ is locally Lipschitz if for each $\x\in\X$ one can find a neighborhood $U$ of $\x$ and a constant $L_U>0$ such that, for every $\y,\z\in U$, the inequality
	$$	|f(\y)-f(\z)| \leq L_U \|\y-\z\|_{\X}$$
	is satisfied.
\end{definition}

For a locally Lipschitz function, we recall the definitions of Clarke's generalized directional derivative and generalized gradient.

\begin{definition}[{\cite[Definition 5.6.3]{ZdSm1}}]
	Consider a locally Lipschitz mapping $f:\X \to \R$.  
	For $\x\in\X$ and $\bv\in\X$, the \emph{generalized directional derivative} of $f$ at $\x$ along $\bv$ is defined by
	$$f^0(\x;\bv)=\lim_{\y\to\x}\sup_{\lambda \downarrow 0}\frac{f(\y+\lambda \bv)-f(\y)}{\lambda}.$$
	The associated \emph{generalized gradient} (also known as Clarke subdifferential) is the subset of $\X^*$ given by
	$$\partial f(\x)=\{\zeta\in\X^{*}: f^0(\x;\bv)\geq \langle \zeta,\bv\rangle \;\; \text{for all } \bv\in\X\}.$$
	We call locally Lipschitz function $f$  \emph{Clarke regular} at $\x$ whenever the one-sided directional derivative $f'(\x;\bv)$ exists for each $\bv\in\X$ and agrees with $f^0(\x;\bv)$.
\end{definition}

	The following results are used in the sequel. 
\begin{proposition}[{\cite[Proposition 2.1.2]{FHC}, \cite[Proposition 5.6.9]{ZdSm1}}]\label{prop-clarke}
	Assume $f:\X\to\R$ is locally Lipschitz. Then the following hold:
	\begin{enumerate}
		\item[(i)] The subdifferential $\partial f(\x)$ is a nonempty, convex and weak$^*$-compact subset of $\X^{*}$, and each element $\boldsymbol{\zeta}\in\partial f(\x)$ satisfies the estimate $\|\boldsymbol{\zeta}\|_{\X^{*}}\leq L_U$.
		\item[(ii)] For every $\x,\bv\in\X$, the generalized directional derivative can be expressed as
		\[
		f^0(\x;\bv) = \max_{\boldsymbol{\zeta}\in\partial f(\x)} \langle \boldsymbol{\zeta}, \bv \rangle.
		\]
	\end{enumerate}
\end{proposition}


\begin{proposition}[{\cite[Proposition 3.23]{SMAOMS}}]\label{prop-sub-diff}
	Let $\X$ be a Banach space and $f:\X \to \R$ a locally Lipschitz function. Then the following hold:
	\begin{enumerate}
		\item[(i)] For every $\x \in \X$, the mapping $\bv \mapsto f^0(\x;\bv)$ is both positively homogeneous and subadditive. More precisely,
		\[
		f^0(\x;\lambda \bv) = \lambda f^0(\x;\bv) 
		\quad \text{for all } \lambda \geq 0,\ \bv \in \X,
		\]
		and
		\[
		f^0(\x;\bv_1+\bv_2) \leq f^0(\x;\bv_1) + f^0(\x;\bv_2)
		\quad \text{for all } \bv_1,\bv_2 \in \X.
		\]
		\item[(ii)] The mapping $(\x,\bv) \mapsto f^0(\x;\bv)$ is upper semicontinuous on $\X \times \X$, that is, whenever sequences 
		$\{\x_n\}$, $\{\bv_n\} \subset \X$ satisfy $\x_n \to \x$ and $\bv_n \to \bv$, we have
		\[
		\limsup_{n\to\infty} f^0(\x_n;\bv_n) \leq f^0(\x;\bv).
		\]
	\end{enumerate}
\end{proposition}

For a more comprehensive treatment of the generalized directional derivative and Clarke subdifferential, we refer the reader to \cite{FHC,SMAOMS}.

To handle the convergence of the nonlinear term, we make use of the following classical result.  

\begin{lemma}[Brezis--Lions Lemma, {\cite[Lemma 1.3]{JLL}}]\label{Lem-Lions}
	Let $\mathcal{O}$ be a bounded open subset of $\R^n$.  
	Suppose $\{\varphi_m\}_{m\in\N}$ and $\varphi$ are functions in $\mathrm{L}^q(\mathcal{O})$ with $1<q<\infty$, satisfying
	$$\sup_{m\in\N} \|\varphi_m\|_{\mathrm{L}^q(\mathcal{O})} \leq C, 
	\qquad \varphi_m \to \varphi \ \text{almost everywhere in } \mathcal{O} \text{ as } m\to\infty.$$
	Then $\varphi_m$ converges weakly to $\varphi$ in $\mathrm{L}^q(\mathcal{O})$ as $m\to\infty$.
\end{lemma}

	\subsection{The model}\label{sub-sec-1}	
Consider a bounded and connected open set $\mathcal{O}\subset \mathbb{R}^d$ with $d\in\{2,3\}$, whose boundary $\Gamma$ is assumed to be Lipschitz. 
A point in $\mathcal{O}$ or on $\Gamma$ will be denoted by $\x$. 
Let $\mathbb{S}^d$ represent the collection of all $d \times d$ symmetric matrices. 
On the spaces $\mathbb{R}^d$ and $\mathbb{S}^d$, we make use of the standard scalar products defined as
\begin{align}
	\bu\cdot\bv&=u_iv_i,\ \bu,\bv\in\R^d,\\
	\boldsymbol{\sigma}\cdot\boldsymbol{\tau}&=\sigma_{ij}\tau_{ij},\ \boldsymbol{\sigma},\boldsymbol{\tau} \in\mathbb{S}^d,
\end{align}
where the Einstein summation rule over repeated indices is employed.

We deal with the two- and three-dimensional forms of the convective Brinkman--Forchheimer extended Darcy (CBFeD) system, which serves as a mathematical model for steady motion of incompressible viscous fluids through porous media:
	\begin{equation}\label{eqn-CBF}
		\left\{
		\begin{aligned}
			-\mu\Delta\bu + (\bu \cdot \nabla)\bu + \alpha\bu + \beta|\bu|^{r-1}\bu + \kappa|\bu|^{q-1}\bu + \nabla p &= \f, \quad \text{in } \mathcal{O}, \\
			\mathrm{div\ }\bu &= 0, \quad \text{in } \mathcal{O}.
		\end{aligned}
		\right.
	\end{equation}
In the above system, $\bu(\cdot):\mathcal{O}\to\mathbb{R}^d$ stands for the velocity field, 
$p(\cdot):\mathcal{O}\to\mathbb{R}$ denotes the pressure, and 
$\f(\cdot):\mathcal{O}\to\mathbb{R}^d$ represents the body force. 
The parameter $\mu>0$ corresponds to the Brinkman coefficient, which reflects the effective viscosity. 
Furthermore, the positive constants $\alpha$ and $\beta$ are associated with the Darcy and Forchheimer damping contributions, 
describing the drag effects related to permeability and porosity of the porous matrix.
 An additional nonlinear term $\kappa|\bu|^{q-1}\bu$ is incorporated to capture a potential pumping mechanism, especially relevant when $\kappa < 0$, which opposes the damping term $\alpha\bu + \beta|\bu|^{r-1}\bu$. Throughout this work, we assume $\kappa < 0$. The exponent $r$, referred to as the absorption exponent, lies in the range $r \in [1, \infty)$, with $r = 3$ identified as the \emph{critical exponent} (\cite[Proposition 1.1.]{KWH}). Additionally, the parameter $q$ satisfies $q \in [1, r)$.	If $\alpha = \beta = \kappa = 0$, the system reduces to the classical Navier-Stokes equations (NSE). When $\alpha, \beta > 0$ and $\kappa = 0$, the system becomes a damped variant of the NSE. The CBFeD model stems from an extended Darcy-Forchheimer law, as discussed in \cite{SGKKMTM,MTT}. The second equation in \eqref{eqn-CBF} enforces the incompressibility condition of the fluid flow.

The CBFeD system presented in \eqref{eqn-CBF} is accompanied by boundary conditions that complete the formulation of the problem. We suppose that the boundary $\Gamma$ of $\mathcal{O}$ is decomposed into two mutually disjoint parts, 
$\Gamma_0$ and $\Gamma_1$, both having positive surface measure. 
Let $\boldsymbol{n}=(n_1,\dots,n_d)$ denote the unit outward normal on $\Gamma$. 
For a vector field $\bu$ defined on the boundary, we distinguish its components as follows: 
the normal part $u_n=\bu\cdot\boldsymbol{n}$ and the tangential component $\bu_\tau=\bu-u_n\boldsymbol{n}$.

For a given velocity field $\bu$ and pressure $p$, the strain-rate tensor is expressed as
$$
\boldsymbol{\varepsilon}(\boldsymbol{u}) = \frac{1}{2}(\nabla \boldsymbol{u} + (\nabla \boldsymbol{u})^\top),
$$
while the associated Cauchy stress tensor takes the form
$$
\boldsymbol{\sigma} = 2\mu \boldsymbol{\varepsilon}(\boldsymbol{u}) - p \I,
$$
with $\I$ denoting the identity matrix. 
On the boundary $\Gamma$, we further decompose the stress vector $\boldsymbol{\sigma} \boldsymbol{n}$ into its normal and tangential components as follows:
$$
\sigma_n = \boldsymbol{n} \cdot \boldsymbol{\sigma} \boldsymbol{n}, \quad \boldsymbol{\sigma}_\tau = \boldsymbol{\sigma} \boldsymbol{n} - \sigma_n \boldsymbol{n}.
$$
The identities $\boldsymbol{u} \cdot \boldsymbol{v} = u_n v_n + \boldsymbol{u}_\tau \cdot \boldsymbol{v}_\tau$ and $(\boldsymbol{\sigma} \boldsymbol{n}) \cdot \boldsymbol{v} = \sigma_n v_n + \boldsymbol{\sigma}_\tau \cdot \boldsymbol{v}_\tau$ are particularly useful in deriving the associated hemivariational inequality. The boundary conditions imposed on the system are
\begin{align}
	\bu&=\boldsymbol{0}\ \text{ on }\ \Gamma_0,\label{eqn-boundary-1}\\
	u_n&=0,\ -\boldsymbol{\sigma}_{\tau}\in\partial j(\bu_{\tau})\ \text{ on }\ \Gamma_1, \label{eqn-boundary-2}
\end{align}
where $j(\bu_\tau)$ is a shorthand for $j(\x,\bu_\tau)$. 
Here, the mapping $j:\Gamma_1\times\R^d \to \R$ is called the \emph{superpotential}, and it is assumed to be locally Lipschitz in its second variable. 
The notation $\partial j$ stands for the \emph{Clarke subdifferential} of the function $j(\x,\cdot)$, defined in the framework of Clarke’s generalized gradient.  Condition \eqref{eqn-boundary-2} encodes a slip-type boundary law. 
If, in addition, $j(\x,\cdot)$ happens to be convex in the second argument, then the coupled system \eqref{eqn-CBF}--\eqref{eqn-boundary-2} reduces to the well-known formulation of a \emph{variational inequality}. 
In contrast, since we do not require convexity of $j$ in this work, the problem naturally falls within the more general class of \emph{hemivariational inequalities}, a setting that is particularly suited for describing processes governed by nonsmooth and nonconvex energy functionals.

	
\subsection{Functional setting} 
In order to cast the system \eqref{eqn-CBF}--\eqref{eqn-boundary-2} into a weak (variational) form, 
we first introduce suitable function spaces. 
For the velocity field, we define
\begin{equation*}
	\mathcal{V} := \Big\{ \bv \in \mathrm{H}^1(\mathcal{O};\mathbb{R}^d) : 
	\bv = \boldsymbol{0} \text{ on } \Gamma_0, \ v_n = 0 \text{ on } \Gamma_1 \Big\}.
\end{equation*}
By employing Korn's inequality and the condition that $|\Gamma_0|>0$ (cf. \cite[Lemma 6.2]{NKJTO}), we have
\begin{align}\label{eqn-equiv-1}
	\|\bv\|_{\mathrm{H}^1(\mathcal{O};\mathbb{R}^d)} \leq C_k \, \|\boldsymbol{\varepsilon}(\bv)\|_{\mathrm{L}^2(\mathcal{O};\mathbb{S}^d)}\,, 
\end{align}
which ensures that $\mathcal{V}$ is a Hilbert space when endowed with the inner product
\begin{align*}
	(\bu,\bv)_{\mathcal{V}} = (\boldsymbol{\varepsilon}(\bu),\boldsymbol{\varepsilon}(\bv))_{\mathrm{L}^2(\mathcal{O};\mathbb{S}^d)}.
\end{align*}
The norm associated with $\mathcal{V}$ is defined by
\begin{align*}
	\|\bu\|_{\mathcal{V}} = \|\boldsymbol{\varepsilon}(\bu)\|_{\mathrm{L}^2(\mathcal{O};\mathbb{S}^d)} 
	= \Bigg(\sum_{i,j=1}^d \int_{\mathcal{O}} |\varepsilon_{ij}(\bu)|^2 \, \d\x \Bigg)^{1/2},
\end{align*}
which is equivalent to the standard $\mathbb{H}^1(\mathcal{O})$ norm. 
For the reduced problem, we focus on the divergence-free subspace of $\mathcal{V}$ given by
\begin{align*}
	\V_{\sigma} := \{ \bv \in \mathcal{V} : \mathrm{div}\,\bv = 0 \ \text{in } \mathcal{O} \}.
\end{align*}
We introduce the space $\mathcal{H} := \mathrm{L}^2(\mathcal{O};\mathbb{R}^d) \equiv \mathbb{L}^2(\mathcal{O})$. 
With this definition, the following continuous embeddings hold:
\begin{align*}
	\mathcal{V} \hookrightarrow \mathcal{H} \equiv \mathcal{H}^* \hookrightarrow \mathcal{V}^*, 
\end{align*}
and
\begin{align*}
	\V_{\sigma} \hookrightarrow \mathcal{H} \equiv \mathcal{H}^* \hookrightarrow \V_{\sigma}^*,
\end{align*}
both of which are dense and compact. Furthermore, let us denote
\[
\mathcal{V}_0 := \mathbb{H}_0^1(\mathcal{O}) \equiv \mathrm{H}_0^1(\mathcal{O};\R^d),
\]
and define the corresponding divergence-free subspace
$$\V_{\sigma,0} := \{ \bv \in \mathcal{V}_0 : \mathrm{div}\,\bv = 0 \ \text{in } \mathcal{O} \}.$$
For the pressure field $p$, we introduce the space
\begin{align*}
	Q := \Big\{ q \in \mathrm{L}^2(\mathcal{O};\mathbb{R}) : 
	(q,1)_{\mathrm{L}^2(\mathcal{O})} := \int_{\mathcal{O}} q(\x) \, d\x = 0 \Big\},
\end{align*}
equipped with the norm $\|q\|_Q := \|q\|_{\mathrm{L}^2(\mathcal{O})}.$

	For the pressure $p$, we define the space 
	
	\subsection{Bilinear, trilinear and nonlinear forms}
We define three bilinear forms, a single trilinear form, and a pair of nonlinear forms, which will play a central role in the weak formulation of the problem. 
For any $\bu, \bv, \bw \in \mathcal{V}$ and $q \in Q$, these forms are given by
\begin{align*}
	a(\bu,\bv) &= \int_{\mathcal{O}} 2\,\boldsymbol{\varepsilon}(\bu) : \boldsymbol{\varepsilon}(\bv) \, \d\x, 
	& a_0(\bu,\bv) &= \int_{\mathcal{O}} \bu \cdot \bv \, \d\x,\\
	d(\bv,q) &= -\int_{\mathcal{O}} q\, \mathrm{div}\,\bv \, \d\x,\\
	b(\bu,\bv,\bw) &= \int_{\mathcal{O}} (\bu \cdot \nabla)\bv \cdot \bw \, \d\x,\\
	c(\bu,\bv) &= \int_{\mathcal{O}} |\bu|^{r-1} \bu \cdot \bv \, \d\x, 
	& c_0(\bu,\bv) &= \int_{\mathcal{O}} |\bu|^{q-1} \bu \cdot \bv \, \d\x.
\end{align*}
In addition, for any $\f \in \L^2(\mathcal{O}) = \mathrm{L}^2(\mathcal{O};\mathbb{R}^d)$, we define the inner product
\begin{align*}
	(\f,\bv)_{\L^2(\mathcal{O})} := \int_{\mathcal{O}} \f \cdot \bv \, \d\x,
\end{align*}
while for $\f \in \V^*$, the duality pairing with $\bv \in \V$ is denoted by 
$\langle \f, \bv \rangle.$

	
It is important to observe that the bilinear form $a(\cdot,\cdot)$ enjoys both boundedness and coercivity on the space $\mathcal{V}$. In other words, $a(\bu,\bv)$ is continuous and satisfies the ellipticity property for every $\bu,\bv \in \mathcal{V}$. More precisely, one has:
%
%
	\begin{align}
		|a(\bu,\bv)|	&\leq 2\|\bu\|_{\V}\|\bv\|_{\V},\ \text{ for all }\ \bu,\bv\in\V,\label{eqn-a-est-1}\\
		a(\bu,\bu)&=2\|\bu\|_{\V}^2, \ \text{ for all }\ \bu\in\V. \label{eqn-a-est-2}
	\end{align}
By combining the Gagliardo-Nirenberg inequality \cite[Theorem 1, pp.~11--12]{LNi} with Korn's inequality (see \eqref{eqn-equiv-1}), it follows that the trilinear form $b(\cdot,\cdot,\cdot)$ is continuous, that is, bounded in $\mathcal{V}$:
	\begin{align}\label{eqn-b-est-1}
		|b(\bu,\bv,\bw)|&\leq \|\bu\|_{\mathbb{L}^4(\mathcal{O})}\|\nabla\bv\|_{\mathbb{L}^2(\mathcal{O})}\|\bw\|_{\mathbb{L}^4(\mathcal{O})}\nonumber\\&\leq C_k C_g^2\|\bu\|_{\H}^{1-\frac{d}{4}}\|\bu\|_{\H^1}^{\frac{d}{4}}\|\bv\|_{\V}\|\bw\|_{\H}^{1-\frac{d}{4}}\|\bw\|_{\H^1}^{\frac{d}{4}}\nonumber\\&\leq C_b\|\bu\|_{\V}\|\bv\|_{\V}\|\bw\|_{\V},
	\end{align}
	where $$C_b=C_k^3C_g^2,$$ $C_g$ denotes the constant arising from the Gagliardo--Nirenberg inequality, while $C_k$ is the constant introduced in \eqref{eqn-equiv-1}.
	 Moreover, the divergence free and boundary conditions of $\bu\in\V$ yield
	\begin{align}
		&b(\bu,\bv,\bw)=-b(\bu,\bw,\bv)\ \text{ for all }\bu,\bv,\bw\in\V, \label{eqn-b-est-3}\\
		&b(\bu,\bv,\bv)=0\ \text{ for all }\bu,\bv\in\V. \label{eqn-b-est-4}
	\end{align} 
	The nonlinear form $c(\cdot,\cdot)$ is continuous when acting on 
	$\L^{r+1}(\mathcal{O}) \times \L^{r+1}(\mathcal{O})$, because
	\begin{align}\label{eqn-c-est}
		|c(\bu,\bv)|\leq\|\bu\|_{\L^{r+1}}^r\|\bv\|_{\L^{r+1}}. 
	\end{align}
	Furthermore, we infer 
	\begin{align}\label{eqn-c-est-1}
		c(\bu,\bu)=\|\bu\|_{\L^{r+1}}^{r+1}. 
	\end{align}
	From \cite[Section 2.4]{SGMTM}, we infer 
	for all 	$\bu,\bv\in\L^{r+1}(\mathcal{O})$ and $r\geq 1$ that 
	\begin{align}\label{2.23}
		&\langle\bu|\bu|^{r-1}-\bv|\bv|^{r-1},\bu-\bv\rangle\geq \frac{1}{2}\||\bu|^{\frac{r-1}{2}}(\bu-\bv)\|_{\H}^2+\frac{1}{2}\||\bv|^{\frac{r-1}{2}}(\bu-\bv)\|_{\H}^2\geq 0,
	\end{align}
	and 
	\begin{align}\label{Eqn-mon-lip}
		&\langle\bu|\bu|^{r-1}-\bv|\bv|^{r-1},\bu-\bv\rangle
		\geq \frac{1}{2^{r-1}}\|\bu-\bv\|_{\L^{r+1}}^{r+1}.
	\end{align}
	For $\mathcal{C}(\bu)=|\bu|^{r-1}\bu$ such that $c(\bu,\bv)=\langle\mathcal{C}(\bu),\bv\rangle$ for all $\bv\in{\L}^{r+1}$,  the Gateaux derivative is given by 
	\begin{align}\label{Gaetu}
		\mathcal{C}'(\bu)\bv&=\left\{\begin{array}{cl}\bv,&\text{ for }r=1,\\ \left\{\begin{array}{cc}|\bu|^{r-1}\bv+(r-1)\frac{\bu}{|\bu|^{3-r}}(\bu\cdot\bv),&\text{ if }\bu\neq \boldsymbol{0},\\\boldsymbol{0},&\text{ if }\bu=\boldsymbol{0},\end{array}\right.&\text{ for } 1<r<3,\\ |\bu|^{r-1}\bv+(r-1)\bu|\bu|^{r-3}(\bu\cdot\bv), &\text{ for }r\geq 3.\end{array}\right.
	\end{align}
	Similar estimates hold for $c_0(\bu,\bv)$ also. 
	
	The bilinear form $d(\cdot,\cdot)$ is bounded in $\V\times Q$, since 
	\begin{align}
		|d(\bv,q)|\leq C\|\bv\|_{\V}\|q\|_Q \ \text{ for all }\ \bv\in\V\ \text{ and }\ q\in Q. 
	\end{align}
	
	\subsection{Assumptions on the superpotential}
	Concerning the superpotential $j$, we shall work under the following set of hypotheses:
	\begin{hypothesis}\label{hyp-sup-j}
	The mapping 	$j:\Gamma_1\times\R^d\to \R$ is such that 
		\begin{itemize}
			\item[(H1)] for every $\boldsymbol{\xi}\in\R^d$, the mapping 
			$\x \mapsto j(\x,\boldsymbol{\xi})$ is measurable on $\Gamma_1$, and in addition 
			$j(\cdot,\mathbf{0}) \in \mathbb{L}^1(\Gamma_1)$;
			
			\item[(H2)] for almost every $\x \in \Gamma_1$, the function 
			$\boldsymbol{\xi} \mapsto j(\x,\boldsymbol{\xi})$ is locally Lipschitz on $\R^d$;
			
			\item[(H3)] there exist constants $k_0,k_1 \geq 0$ such that, for almost every 
			$\x\in\Gamma$ and all $\boldsymbol{\xi}\in\R^d$, every 
			$\boldsymbol{\eta}\in \partial j(\x,\boldsymbol{\xi})$ satisfies
			\[
			|\boldsymbol{\eta}| \leq k_0 + k_1|\boldsymbol{\xi}|;
			\]
			
			\item [(H4)] $(\boldsymbol{\eta}_1-\boldsymbol{\eta}_2)\cdot(\boldsymbol{\xi}_1-\boldsymbol{\xi}_2)\geq -\delta_1|\boldsymbol{\xi}_1-\boldsymbol{\xi}_2|^2$ for all $\boldsymbol{\xi}_i\in\R^d$, $\boldsymbol{\eta}_i\in \partial j(\x,\boldsymbol{\xi}_i)$, $i=1,2,$ for a.e. $\x\in\Gamma$  with $\delta_1\geq 0$. 
		\end{itemize}
	\end{hypothesis}
	
	By Hypothesis \ref{hyp-sup-j} (H3) and Proposition \ref{prop-clarke} (ii), we have 
	\begin{align}\label{eqn-j0-est}
		|j^0(\x,\boldsymbol{\xi}_1;\boldsymbol{\xi}_2)|\leq \left(k_0+k_1|\boldsymbol{\xi}_1|\right)|\boldsymbol{\xi}_2|\ \text{ for all }\ \boldsymbol{\xi}_1, \boldsymbol{\xi}_2\in\R^d, \ \text{ a.e. }\ \x\in\Gamma_1. 
	\end{align}
	
	Assumption (H4) is usually described in the literature as a relaxed monotonicity requirement 
	(see \cite[Definition 3.49]{SMAOMS}). An equivalent formulation can be stated as:
	\begin{align}\label{eqn-alternative}
		j^0(\x,\boldsymbol{\xi}_1;\boldsymbol{\xi}_2-\boldsymbol{\xi}_1)+	j^0(\x,\boldsymbol{\xi}_2;\boldsymbol{\xi}_1-\boldsymbol{\xi}_2)\leq\delta_1|\boldsymbol{\xi}_1-\boldsymbol{\xi}_2|^2\ \text{ for all }\ \boldsymbol{\xi}_1,\boldsymbol{\xi}_2\in\R^d. 
	\end{align}
	We now introduce the functional $J:\L^2(\Gamma)\to\R$ given by
	\begin{align}\label{eqn-J-def}
		J(\bv)=\int_{\Gamma_1}j(\x, \bv_{\tau}(\x))\d S,\ \bv\in\L^2(\mathcal{O}). 
	\end{align}

	The next lemma, inspired by \cite[Lemma 13]{SMAO} and \cite[Lemma 6.2]{CFWH}, is stated here in a slightly modified form.

	\begin{lemma}\label{lem-hemi-var}
	Suppose that $j:\Gamma_1\times\R\to\R$ fulfills Hypothesis~\ref{hyp-sup-j}. Then the functional $J$ introduced in \eqref{eqn-J-def} enjoys the following properties:
		\begin{enumerate}
			\item $J(\cdot)$ is locally Lipschitz in $\mathbb{L}^2(\Gamma_1)$. 
			\item $\|\z\|_{\mathbb{L}^2(\Gamma_1)}\leq k_0|\Gamma_1|^{1/2}+k_1\|\bv\|_{\L^2(\Gamma_1)}$ for all $\bv\in\L^2(\Gamma_1)$, $\z\in\partial J(\bv)$ with $k_0,k_1\geq 0$,
			\item $J^0(\bu;\bv)\leq\int_{\Gamma_1}j^0(\bu_{\tau}(\x);\bv_{\tau}(\x))\d S$ for all $\bu,\bv\in\mathbb{L}^2(\Gamma_1)$. 
			\item $(\z_1-\z_2,\bu_1-\bu_2)_{\mathbb{L}^2(\Gamma_1)}\geq -\delta_1\|\bu_1-\bu_2\|_{\mathbb{L}^2(\Gamma_1)}^2$ for all $\z_i\in\partial J(\bu_i)$, $\bu_i\in \L^2(\Gamma_1)$, $i=1,2$, with $\delta_1\geq 0$. 
		\end{enumerate}
	\end{lemma}

	\section{Problem formulation and the CBFeD Hemivariational Inequality}\label{sec3}\setcounter{equation}{0}
	In order to obtain the weak formulation of system \eqref{eqn-CBF}--\eqref{eqn-boundary-2}, 
	we first recast the primary equation \eqref{eqn-CBF} in the form:
	\begin{align}\label{eqn-concrete-new}
		-2\mu\mathrm{div}(\boldsymbol{\varepsilon}(\bu))+(\bu\cdot\nabla\bu)+\alpha\bu+\beta|\bu|^{r-1}\bu+\kappa|\bu|^{q-1}\bu+\nabla p=\f, \ \text{ in }\ \mathcal{O}.
	\end{align}
	Suppose that the system \eqref{eqn-CBF}--\eqref{eqn-boundary-2} admits a sufficiently smooth solution 
	$(\bu,p)$, ensuring that the calculations in the following derivation are well-defined: We take the inner product of equation \eqref{eqn-concrete-new} with an arbitrary smooth test function $\bv\in\V\cap\L^{r+1}(\mathcal{O})$ to find 
	\begin{align}
		&\int_{\mathcal{O}}\left[-2\mu\mathrm{div}(\boldsymbol{\varepsilon}(\bu))\cdot\bv+(\bu\cdot\nabla\bu)\cdot\bv+\alpha\bu\cdot\bv+\beta|\bu|^{r-1}\bu\cdot\bv+\kappa|\bu|^{q-1}\bu\cdot\bv+\nabla p\right]\d\x \nonumber\\&\quad=\int_{\mathcal{O}}\f\cdot\bv\d\x. 
	\end{align}
	By carrying out integration by parts, we deduce that
	\begin{align}
		\int_{\mathcal{O}}\Big[2\mu\boldsymbol{\varepsilon}(\bu)\cdot\boldsymbol{\varepsilon}(\bu)+(\bu\cdot\nabla\bu)\cdot\bv+\alpha\bu\cdot\bv&+\beta|\bu|^{r-1}\bu\cdot\bv+\kappa|\bu|^{q-1}\bu\cdot\bv-p\mathrm{div\ }\bv\Big]\d\x\nonumber\\-\int_{\Gamma}\boldsymbol{\sigma}\n\cdot\bv\d S &=\int_{\mathcal{O}}\f\cdot\bv\d\x. 
	\end{align}
	Upon applying the boundary conditions satisfied by $\bv$, it follows that
	\begin{align}
		\mu a(\bu,\bv)+b(\bu,\bu,\bv)+\alpha a_0(\bu,\bv)&+\beta c(\bu,\bv)+\kappa c_0(\bu,\bv)+d(\bv,p)\nonumber\\+\int_{\Gamma_1}(-\boldsymbol{\sigma}_\tau)\cdot\bv_{\tau}\d S&=	(\f,\bv)_{\L^2(\mathcal{O})}. 
	\end{align}
	Invoking the boundary condition \eqref{eqn-boundary-2}, it follows that
	\begin{align}
		-\boldsymbol{\sigma}_\tau\in\partial j(\bu_{\tau})\ \text{ on }\ \Gamma_1, 
	\end{align}
	so that 
	\begin{align}
		\int_{\Gamma_1}(-\boldsymbol{\sigma}_\tau)\cdot\bv_{\tau}\d S\leq \int_{\Gamma_1}j^0(\bu_{\tau};\bv_{\tau})\d S.
	\end{align}
	Consequently, if $\f \in \V^*$, then for every smooth $\bv \in \V \cap \L^{r+1}(\mathcal{O}),$ we obtain
	\begin{align}
		\mu a(\bu,\bv)+b(\bu,\bu,\bv)+\alpha a_0(\bu,\bv)&+\beta c(\bu,\bv)+\kappa c_0(\bu,\bv)+d(\bv,p)\nonumber\\+\int_{\Gamma_1}j^0(\bu_{\tau};\bv_{\tau})\d S&\geq 	\langle\f,\bv\rangle. 
	\end{align}
	In the next step, we multiply the second equation in \eqref{eqn-CBF} by an arbitrary function $q \in Q$ and integrate over $\mathcal{O}$, leading to
	\begin{align}
		d(\bu,q)=0. 
	\end{align}
	
 We thus arrive at the hemivariational inequality associated with problem \eqref{eqn-CBF}–\eqref{eqn-boundary-2}, which takes the form:
	\begin{problem}\label{prob-hemi-1}
		Find $\bu\in\V\cap\L^{r+1}(\mathcal{O})$ and $p\in Q$  such that
		\begin{equation}\label{eqn-hemi-1}
			\left\{
			\begin{aligned}
				\mu a(\bu,\bv)+b(\bu,\bu,\bv)+\alpha a_0(\bu,\bv)&+\beta c(\bu,\bv)+\kappa c_0(\bu,\bv)+d(\bv,p)\\+\int_{\Gamma_1}j^0(\bu_{\tau};\bv_{\tau})\d S&\geq 	\langle\f,\bv\rangle, \ \text{ for all }\ \bv\in \V\cap\L^{r+1}(\mathcal{O}),\\
				d(\bu,q)&=0,\ \text{ for all }\ q\in Q. 
			\end{aligned}\right.
		\end{equation}
	\end{problem}
	
	By eliminating the unknown pressure variable $p$, the problem reduces to the following hemivariational inequality:
	\begin{problem}\label{prob-hemi-2}
		Find $\bu\in\V_{\sigma}\cap\L^{r+1}(\mathcal{O})$ such that 
		\begin{equation}\label{eqn-hemi-2}
			\left\{
			\begin{aligned}
				\mu a(\bu,\bv)+b(\bu,\bu,\bv)+\alpha a_0(\bu,\bv)&+\beta c(\bu,\bv)+\kappa c_0(\bu,\bv)+\int_{\Gamma_1}j^0(\bu_{\tau};\bv_{\tau})\d S\\&\geq 	\langle\f,\bv\rangle, \ \text{ for all }\ \bv\in \V_{\sigma}\cap\L^{r+1}(\mathcal{O}).
			\end{aligned}\right.
		\end{equation}
	\end{problem}
	
	The forthcoming section is devoted to discuss the existence and uniqueness of solutions for Problems~\ref{prob-hemi-1} and \ref{prob-hemi-2} 
	(see Theorems~\ref{thm-main-hemi}, \ref{thm-unique}, and \ref{thm-equivalent}).

	\subsection{The CBFeD Hemivariational Inequality}
	We begin by analyzing Problem~\ref{prob-hemi-2}. 
	To ensure the existence of a solution, we assume the following smallness condition:
	\begin{align}\label{eqn-cond}
		k_1<2\mu\lambda_0,
	\end{align}
	where $k_1$ corresponds to the constant appearing in Hypothesis~\ref{hyp-sup-j}, and $\lambda_0$ denotes the smallest eigenvalue of the related eigenvalue problem:
	\begin{align}
		\bu\in\V, \ \int_{\mathcal{O}}\boldsymbol{\varepsilon}(\bu):\boldsymbol{\varepsilon}(\bv)\d\x =\lambda\int_{\Gamma_1}\bu_{\tau}\cdot\bv_{\tau}\d S\ \text{ for all }\ \bv\in\V. 
	\end{align}
	Because the trace mapping $\bu \mapsto \bu_{\tau}\big|_{\Gamma_1}$ is compact from $\V$ into $\L^2(\Gamma_1)$ 
	(due to the compact embedding $\H^1(\mathcal{O}) \hookrightarrow \L^2(\Gamma)$), the spectral theory for compact self-adjoint operators guarantees the existence of a sequence of eigenvalues 
	$\{\lambda_k\}_{k=0}^{\infty}$ satisfying $ \lambda_k > 0$ and $\lambda_k \to +\infty$.  
	Furthermore, the following trace inequality holds:
	\begin{align}\label{eqn-trace}
		\|\bv_{\tau}\|_{\L^2(\Gamma_1)}\leq\lambda_0^{-1/2}\|\bv\|_{\V} \ \text{ for all }\ \bv\in\V. 
	\end{align}

	The following result provides the existence of a solution to Problem \ref{prob-hemi-2}. For the solvability of Problem \ref{prob-hemi-2}, we emphasize that the parameters are taken as $r \in [1, \infty)$ and $q \in [1, r)$, with no additional restrictions imposed on the exponent $r$.

	\begin{theorem}[{\cite[Theorem 3.1]{WAMTM}}]\label{thm-main-hemi}
		Under the assumptions of Hypothesis~\ref{hyp-sup-j} (H1)–(H3) together with \eqref{eqn-cond}, 
		Problem~\ref{prob-hemi-2} admits at least one solution $\bu \in \V_{\sigma} \cap \L^{r+1}(\mathcal{O})$.
	\end{theorem}

We now establish that each solution of Problem~\ref{prob-hemi-2} remains bounded.
	\begin{proposition}[{\cite[Proposition 3.2]{WAMTM}}]\label{prop-energy-est}
		Assuming the conditions of Theorem~\ref{thm-main-hemi}, let $\bu \in \V_{\sigma} \cap \L^{r+1}$ be a solution of Problem~\ref{prob-hemi-2}. Then
		\begin{align}\label{eqn-bound}
			\|\bu\|_{\V}^2+	\|\bu\|_{\L^{r+1}}^{r+1}\leq 2\max\bigg\{\frac{1}{\left(2\mu-k_1\lambda_0^{-1}\right)},\frac{1}{\beta}\bigg\} K_{\f}=:\widetilde{K}_{\f},
		\end{align}
		where
		\begin{align}\label{eqn-value-k} K_{\f}=\frac{1}{\left(2\mu-k_1\lambda_0^{-1}\right)}\left(\|\f\|_{\V^*}+k_0|\Gamma_1|^{1/2}\lambda_0^{-1/2}\right)^2+ |\kappa|^{\frac{r+1}{r-q}},
		\end{align}
		and  the constants $k_0$ and $k_1$ are from Hypothesis \ref{hyp-sup-j} (H3). 
	\end{proposition}

	We now turn our attention to establishing the uniqueness of the solution.
	
	\begin{theorem}[{\cite[Theorem 3.3]{WAMTM}}]\label{thm-unique}
		Assume Hypothesis~\ref{hyp-sup-j} (H1)–(H4) holds and let $r \in (3, \infty)$. Suppose that either
		\begin{align}\label{eqn-unique-con-1}
			\mbox{$\mu>\frac{\delta_1}{2\lambda_0}$\  and \ $\alpha\geq ({\varrho}_{1,r}+\varrho_{2,r}+\varrho_{3,r})$,}
		\end{align}
		where 
			\begin{align}\label{eqn-rho-2}
			\varrho_{i,r}=\left(\frac{r-q}{r-1}\right)\left(\frac{2i(q-1)}{\beta(r-1)}\right)^{\frac{q-1}{r-q}}\left( |\kappa| q2^{q-1}\right)^{\frac{r-1}{r-q}},\ i=1,2,
		\end{align}
		and 
		\begin{align}\label{eqn-varrhotilde}
			{\varrho}_{3,r}=\left(\frac{C_k^2}{2(2\mu-\delta_1\lambda_0^{-1})}\right)^{\frac{r-1}{r-3}}\left(\frac{r-3}{r-1}\right)\left(\frac{8}{\beta (r-1)}\right)^{\frac{2}{r-3}}
		\end{align}
		or 
		\begin{align}\label{eqn-unique-con-11}
			\mbox{ $\mu>\frac{\delta_1}{2\lambda_0}$, $\alpha\geq (\varrho_{1,r}+\varrho_{2,r}+\widehat{\varrho}_{3})$\  and \ $\beta\geq 4\widehat{\varrho}_{3}$,}
		\end{align}
		where 
		\begin{align}\label{eqn-varrhohat}
			\widehat{\varrho}_{3}=\frac{C_k^2}{2(2\mu-\delta_1\lambda_0^{-1})}
		\end{align}
		holds.	Then Problem  \ref{prob-hemi-2} has a unique solution. 
		
		For $r\in[1,3]$, assume that 
		\begin{align}\label{eqn-unique-con-2}
			\mbox{$\mu>\frac{\delta_1}{2\lambda_0}$\  and \  $\alpha>	\widehat{\varrho}_4(C_gC_k)^{\frac{8}{4-d}}\widetilde{K}^{\frac{4}{4-d}}+\varrho_{1,r}+\varrho_{2,r}$,}
		\end{align}
		where 
		\begin{align}\label{eqn-varrhohat-1}
			\widehat{\varrho}_4=C_g^{\frac{8}{4-d}}C_k^{\frac{8}{4-d}}\left(\frac{8}{4-d}\right)\left(\frac{4+d}{2(2\mu-\delta_1\lambda_0^{-1})}\right)^{\frac{4+d}{4-d}}
		\end{align}
		and $\widetilde{K}$ is defined in \eqref{eqn-bound}, then Problem~\ref{prob-hemi-2} possesses a unique solution. 
		
		Moreover, the mapping $\boldsymbol{f} \in \V^* \mapsto \boldsymbol{u} \in \V_{\sigma} \cap \L^{r+1}$ is
		
		\begin{align*}
			\left\{
		\begin{array}{ll}
		\mbox{Lipschitz continuous,}&  \mbox{for $d=2$ with $r\in[1,\infty)$  and for $d=3$ with $r\in[1,5]$,} \\
		\mbox{H\"older continuous,}& \mbox{for $d=3$ with $r\in(5,\infty)$.}
		\end{array}
		\right.
		\end{align*}
	\end{theorem}

	\begin{remark}
		If one considers convective Brinkman-Forchheimer (CBF) HVI, that is, Problem \ref{prob-hemi-1} with $\kappa=0$, then $\varrho_{i,r}=0$ for $i=1,2$ in \eqref{eqn-unique-con-1}, \eqref{eqn-unique-con-11} and \eqref{eqn-unique-con-2}. 
	\end{remark}

We now turn our attention to Problem~\ref{prob-hemi-1}. 
From this point onward, we assume that $1 \le r \le \frac{2d}{d-2}$, with the convention that $1 \le r < \infty$ when $d=2$, which we denote as $1 \le r \le \frac{2d}{(d-2)^+}$. 
Under this assumption, the Sobolev embedding theorem ensures that $\H^1(\mathcal{O}) \hookrightarrow \L^{r+1}(\mathcal{O})$, and thus Problem~\ref{prob-hemi-1} reduces to finding $\bu \in \V$. 
We next state the inf-sup condition (also referred to as the Ladyzhenskaya–Babuška–Brezzi (LBB) condition; see \cite[Theorem 2.2]{JSHNJW}):

		\begin{align}\label{eqn-inf-sup}
			\vartheta\|q\|_{\mathrm{L}^2}\leq 	\sup_{\bv\in\V_0}\frac{d(\bv,q)}{\|\bv\|_{\V}}, \ \text{ for all }\ q\in Q,
		\end{align}
		The above condition is used  to prove the next result.
		\begin{theorem}[{\cite[Theorem 3.5]{WAMTM}}]\label{thm-equivalent}
			Assuming $1 \le r \le \frac{2d}{(d-2)^+}$ and the conditions of Theorem~\ref{thm-unique}, 
			Problem~\ref{prob-hemi-1} possesses a unique solution.
			
			Furthermore, $p\in Q$ depends locally Lipschitz continuously on $\f\in\V^*$. 
		\end{theorem}

		\section{A stability result}\setcounter{equation}{0} \label{sec4}
		For further analysis, we first consider a perturbed stationary CBFeD inequality, in which the external force density $\f$ and the superpotential $j$ are replaced by their respective perturbations $\f_m \in \V^*$ and $j_m$, for $m\in\N$. Motivated from \cite[Section 3.1]{WWXCWH}, similar to Hypothesis \ref{hyp-sup-j}, we make the following assumptions regarding $j_m$: 
		
	\begin{hypothesis}\label{hyp-sup-jm}
		The mapping $j_m:\Gamma_1\times\R^d\to\R$ is such that 
	\begin{itemize}
		\item [(H1)] For all $\boldsymbol{\xi} \in \R^d$, the function $j_m(\cdot, \boldsymbol{\xi})$ is measurable on $\Gamma_1$, and $j_m(\cdot, \boldsymbol{0}) \in \mathbb{L}^1(\Gamma_1)$;
		\item [(H2)] For almost every $\x \in \Gamma_1$, the mapping $\boldsymbol{\xi} \mapsto j_m(\x, \boldsymbol{\xi})$ is locally Lipschitz continuous on $\R^d$;
		\item [(H3)] For almost every $\x \in \Gamma$ and all $\boldsymbol{\xi} \in \R^d$, each $\boldsymbol{\eta} \in \partial j_m(\x, \boldsymbol{\xi})$ satisfies
		\[
		|\boldsymbol{\eta}| \le k_{0,m} + k_{1,m} |\boldsymbol{\xi}|, \quad 0 \le k_{0,m} \le k_0, \ 0 \le k_{1,m} \le k_1;
		\]
		\item [(H4)] For almost every $\x \in \Gamma$ and for all $\boldsymbol{\xi}_i \in \R^d$, $\boldsymbol{\eta}_i \in \partial j_m(\x, \boldsymbol{\xi}_i)$, $i=1,2$, there holds
		\[
		(\boldsymbol{\eta}_1 - \boldsymbol{\eta}_2) \cdot (\boldsymbol{\xi}_1 - \boldsymbol{\xi}_2) \ge -\delta_{1,m} |\boldsymbol{\xi}_1 - \boldsymbol{\xi}_2|^2,
		\]
		where $\delta_{1,m} \ge 0$ and there exists a constant $\delta_1 > 0$ such that $\delta_{1,m} \le \delta_1$ for every $m \in \N$.
	\end{itemize}
	\end{hypothesis}
	Condition (H4) is often called the \emph{relaxed monotonicity condition} in the literature (see \cite[Definition 3.49]{SMAOMS}) and can be equivalently expressed as:
	for a.e. $\x\in\Gamma_1$,
	\begin{align}\label{eqn-alternative-1}
		j^0_m(\x,\boldsymbol{\xi}_1;\boldsymbol{\xi}_2-\boldsymbol{\xi}_1)+	j^0_m(\x,\boldsymbol{\xi}_2;\boldsymbol{\xi}_1-\boldsymbol{\xi}_2)\leq\delta_{1,m}|\boldsymbol{\xi}_1-\boldsymbol{\xi}_2|^2\ \text{ for all }\ \boldsymbol{\xi}_1,\boldsymbol{\xi}_2\in\R^d. 
	\end{align}
		The corresponding perturbed problem is formulated as follows: 
			\begin{problem}\label{prob-hemi-approx}
			Find $\bu_m\in\V$ and $p\in Q$  such that
			\begin{equation}\label{eqn-hemi-approx}
				\left\{
				\begin{aligned}
					\mu a(\bu_m,\bv)+b(\bu_m,\bu_m,\bv)+\alpha a_0(\bu_m,\bv)&+\beta c(\bu_m,\bv)+\kappa c_0(\bu_m,\bv)+d(\bv,p_m)\\+\int_{\Gamma_1}j^0_m(\bu_{m,\tau};\bv_{\tau})\d S&\geq 	\langle\f_m,\bv\rangle, \ \text{ for all }\ \bv\in \V,\\
					d(\bu_m,q)&=0,\ \text{ for all }\ q\in Q. 
				\end{aligned}\right.
			\end{equation}
		\end{problem}
		
	To measure the consistency between the given problem data and the theoretical framework, we impose the following hypothesis on the superpotential.
	\begin{hypothesis}\label{hyp-sup-jm-sup}
		\begin{itemize}
			\item[(H5)] If $\boldsymbol{\xi}_m\to\boldsymbol{\xi}$  and $\boldsymbol{\eta}_m\to\boldsymbol{\eta}$ in $\R^d$, then
			\begin{align}
				\limsup_{m\to\infty}j_m^0(\boldsymbol{\xi}_m;\boldsymbol{\eta}_m)\leq j^0(\boldsymbol{\xi};\boldsymbol{\eta}). 
			\end{align}
		\end{itemize}
	\end{hypothesis}
		Refer to \cite[Example 2.4]{CFWH1} for an example of a superpotential that satisfies Hypotheses \ref{hyp-sup-j}, \ref{hyp-sup-jm}, and \ref{hyp-sup-jm-sup}. To ensure the existence of a solution to the perturbed problem, an additional condition on the force density is necessary. For a given constant 
		$m_0>0$, define a subset $\V_{m_0}^*\subset\V^*$ by
		\begin{align*}
			\V_{m_0}^*=\left\{\f\in\V^*:\|\f\|_{\V^*}\leq m_0\right\}.
		\end{align*}
		Similar to \eqref{eqn-value-k}, we define 
			\begin{align}\label{eqn-value-km0} K_{m_0}=\frac{1}{\left(2\mu-k_1\lambda_0^{-1}\right)}\left(m_0+k_0|\Gamma_1|^{1/2}\lambda_0^{-1/2}\right)^2+ |\kappa|^{\frac{r+1}{r-q}},
		\end{align}
		Then we have the following result. 
		\begin{theorem}\label{thm-conv}
			Let Hypotheses \ref{hyp-sup-j} and \ref{hyp-sup-jm} be satisfied. Then under assumptions of Theorem \ref{thm-unique}, Problem \ref{prob-hemi-1} has a unique solution $(\bu, p)\in\V\times Q$, Problem \ref{prob-hemi-approx} has a unique solution $(\bu_m,p_m) \in\V\times Q$, and
\begin{align}\label{eqn-bound-discrete}
	\|\bu\|_{\V}^2+	\|\bu\|_{\L^{r+1}}^{r+1}\leq \widetilde{K}_{m_0} \ \text{ and }\ \|\bu_m\|_{\V}^2+	\|\bu_m\|_{\L^{r+1}}^{r+1}\leq \widetilde{K}_{m_0},
\end{align}
where $\widetilde{K}_{m_0}=2\max\left\{\frac{1}{\left(2\mu-k_1\lambda_0^{-1}\right)},\frac{1}{\beta}\right\} K_{m_0}$ and $K_{m_0}$ is defined in \eqref{eqn-value-km0}. 

Moreover, if Hypothesis \ref{hyp-sup-jm-sup}  is satisfied and $\|\f_m-\f\|_{\V^*}\to 0$, then 
\begin{align*}
	\bu_m\to \bu \ \text{ in }\ \V\ \text{ and }\ p_m\to p\ \text{ in }\ Q. 
\end{align*}
		\end{theorem}
		
		\begin{proof}
			According to Theorem \ref{thm-equivalent}, under the given assumptions, Problems \ref{prob-hemi-1} and \ref{prob-hemi-approx} each admit a unique solution, denoted by $(\bu, p) \in \V \times Q$ and $(\bu_m, p_m) \in \V \times Q$, respectively. Now, suppose that $\f_m \to \f$ in $\V^{*}$, and  assume Hypothesis \ref{hyp-sup-jm-sup}.  We divide the rest of the proof into the following steps:

			\vskip 0.2cm
			\noindent \textbf{Step 1:} Let us first show that the sequences $\{\|\bu_m\|_{\V}\}_{m\in\N}$ and $\{\|p_m\|_Q\}_{m\in\N}$ are bounded. For $\bv\in\V_0$, by following the standard procedure analogous to the derivation of Problem  \ref{prob-hemi-1}, equations \eqref{eqn-CBF}-\eqref{eqn-boundary-2} yield the following result: 
			\begin{align}\label{eqn-con-1}
					\mu a(\bu,\bv)+b(\bu,\bu,\bv)+\alpha a_0(\bu,\bv)+\beta c(\bu,\bv)+\kappa c_0(\bu,\bv)+d(\bv,p)=	\langle\f,\bv\rangle.
			\end{align}
			Analogously, we obtain the following from the perturbed problem: 
				\begin{align}\label{eqn-con-2}
				\mu a(\bu_m,\bv)+b(\bu_m,\bu_m,\bv)+\alpha a_0(\bu_m,\bv)+\beta c(\bu_m,\bv)+\kappa c_0(\bu_m,\bv)+d(\bv,p_m)=	\langle\f_m,\bv\rangle.
			\end{align}
			Subtracting \eqref{eqn-con-1} from \eqref{eqn-con-2}, we infer for all $\bv\in\V_0$ that 
			\begin{align}\label{eqn-con-3}
		d(\bv,p_m-p)&=\langle\f_m-\f,\bv\rangle+\mu a(\bu-\bu_m,\bv)+b(\bu,\bu,\bv)-b(\bu_m,\bu_m,\bv)\nonumber\\&\quad+\alpha a_0(\bu-\bu_m,\bv)+\beta [c(\bu,\bv)-c(\bu_m,\bv)]+\kappa [c_0(\bu,\bv)-c_0(\bu_m,\bv)].
		\end{align}
		We consider $b(\bu,\bu,\bv)-b(\bu_m,\bu_m,\bv)$ and estimate it using \eqref{eqn-b-est-1} as 
		\begin{align}\label{eqn-con-4}
			&b(\bu,\bu,\bv)-b(\bu_m,\bu_m,\bv)
			\nonumber\\&= b(\bu-\bu_m,\bu,\bv)+b(\bu_m,\bu-\bu_m,\bv)\nonumber\\&\leq C_b(\|\bu\|_{\V}+\|\bu_m\|_{\V})\|\bu-\bu_m\|_{\V}\|\bv\|_{\V}. 
		\end{align}
		Using Taylor's formula, we estimate $c(\bu,\bv)-c(\bu_m,\bv)$ as 
		\begin{align}\label{eqn-con-5}
			|c(\bu,\bv)-c(\bu_m,\bv)|&=\left|\left<\int_0^1\mathcal{C}'(\theta\bu+(1-\theta)\bu_m)(\bu-\bu_m)\d\theta,\bv\right>\right|\nonumber\\&\leq r(\|\bu\|_{\L^{r+1}}+\|\bu_m\|_{\L^{r+1}})^{r-1}\|\bu-\bu_m\|_{\L^{r+1}}\|\bv\|_{\L^{r+1}}\nonumber\\&\leq r C_s^2(\|\bu\|_{\L^{r+1}}+\|\bu_m\|_{\L^{r+1}})^{r-1}\|\bu-\bu_m\|_{\V}\|\bv\|_{\V}.
		\end{align}
		A similar calculation yields 
		\begin{align}\label{eqn-con-6}
				|c_0(\bu,\bv)-c_0(\bu_m,\bv)|&\leq qC_s^2|\mathcal{O}|^{\frac{r-q}{r+1}}(\|\bu\|_{\L^{r+1}}+\|\bu_m\|_{\L^{r+1}})^{q-1}\|\bu-\bu_m\|_{\V}\|\bv\|_{\V}.
		\end{align}
		Taking into account of estimates \eqref{eqn-con-4}-\eqref{eqn-con-6} and the inf-sup condition \eqref{eqn-inf-sup}, we obtain
		\begin{align}\label{eqn-con-7}
			\vartheta\|p_m-p\|_{Q}&\leq\sup_{\bv\in\V_0}\frac{d(\bv,p_m-p)}{\|\bv\|_{\V}}\nonumber\\&= \sup_{\bv\in\V_0}\frac{\begin{array}{l}\langle\f_m-\f,\bv\rangle+\mu a(\bu-\bu_m,\bv)+b(\bu,\bu,\bv)-b(\bu_m,\bu_m,\bv)\\+\alpha a_0(\bu-\bu_m,\bv)+\beta [c(\bu,\bv)-c(\bu_m,\bv)]+\kappa [c_0(\bu,\bv)-c_0(\bu_m,\bv)]\end{array}}{\|\bv\|_{\V}}\nonumber\\&\leq \|\f_m-\f\|_{\V^*}+(2\mu+C_k^2\alpha)\|\bu-\bu_m\|_{\V}+C_b(\|\bu\|_{\V}+\|\bu_m\|_{\V})\|\bu-\bu_m\|_{\V}\nonumber\\&\quad+r C_s^2(\|\bu\|_{\L^{r+1}}+\|\bu_m\|_{\L^{r+1}})^{r-1}\|\bu-\bu_m\|_{\V}\nonumber\\&\quad+qC_s^2|\mathcal{O}|^{\frac{r-q}{r+1}}(\|\bu\|_{\L^{r+1}}+\|\bu_m\|_{\L^{r+1}})^{q-1}\|\bu-\bu_m\|_{\V}.
		\end{align}
		Since $\{\|\bu_m\|_{\V}\}_{m\in\N}$,  $\{\|\bu_m\|_{\L^{r+1}}\}_{m\in\N}$ and  $\{\|\f_m\|_{\V}^*\}_{m\in\N}$ are bounded (see \eqref{eqn-bound-discrete}), it follows from \eqref{eqn-con-7} that $\{\|p_m\|_{Q}\}_{m\in\N}$ is also bounded. 
		
			\vskip 0.2cm
		\noindent \textbf{Step 2:} Let us now prove the following weak convergences: 
		\begin{align}\label{eqn-con-8}
			\bu_m\xrightarrow{w}\bu \ \text{ in }\ \V \ \text{ and }\ p_m\xrightarrow{w} p \ \text{ in }\ Q \ \text{ as }\ m\to\infty. 
		\end{align}
			Since the sequences $\{\|\bu_m\|_{\V}\}_{m\in\N}$ and $\{\|p_m\|_{Q}\}_{m\in\N}$  are uniformly bounded, an application of the Banach-Alaoglu theorem yields the existence of  subsequences of  $\{\|\bu_m\|_{\V}\}_{m\in\N}$ and $\{\|p_m\|_{Q}\}_{m\in\N}$ (still denoted by the same symbol), and two elements $\overline{\bu}\in\V$ and $\overline{p}\in Q$  such that 
			\begin{align}\label{eqn-con-9}
				\bu_m\xrightarrow{w}\overline{\bu} \ \text{ in }\ \V \ \text{ and }\ p_m\xrightarrow{w} \overline{p} \ \text{ in }\ Q \ \text{ as }\ m\to\infty. 
			\end{align}
			Owing to the compactness of the embedding $\mathcal{V} \hookrightarrow \mathcal{H}$, there exists a subsequence (still denoted by the same indices) such that
			
			\begin{align}\label{eqn-con-10}
	\bu_m \to \overline{\bu} \ \text{ strongly in } \ \mathcal{H},
				\end{align}
				and 
					\begin{align}\label{eqn-con-11}
					\bu_m(\x) \to \overline{\bu}(\x) \ \text{ for a.e. } \ \x\in\mathcal{O},
				\end{align}
				along an additional subsequence, which we continue to denote by the same symbol. 
		Moreover, the compact embedding $\mathcal{V} \hookrightarrow \L^2(\Gamma_1)$(see \cite[Theorem 6.2, Chapter 2]{JNe}) yields 
		\begin{align}\label{eqn-con-12}
	\mbox{$\bu_{m,\tau}\to \overline{\bu}_{\tau}$ \ strongly in \ $\L^2(\Gamma_1)$. }
		\end{align}
		By passing to a subsequence (not relabeled), we also have 
			\begin{align}\label{eqn-con-13}
		\mbox{$\bu_{m,\tau} \to \overline{\bu}_{\tau}$ \ a.e. on \ $\Gamma_1$.}
	\end{align}
	Let us now consider
	\begin{align}\label{eqn-con-14}
	&	|b(\bu_m,\bu_m,\bv)-b(\overline{\bu},\overline{\bu},\bv)|\nonumber\\&\leq	|b(\bu_m-\overline{\bu},\bu_m,\bv)|+|b(\overline{\bu},\bu_m-\overline{\bu},\bv)|\nonumber\\&\leq C_k^2C_g^2\|\bu_m-\overline{\bu}\|_{\H}^{1-\frac{d}{4}}\|\bu_m-\overline{\bu}\|_{\V}^{\frac{d}{4}}\|\bu_m\|_{\V}\|\bv\|_{\V}+|b(\overline{\bu},\bu_m-\overline{\bu},\bv)|\nonumber\\&\to 0\ \text{ as }\ m\to\infty, 
	\end{align}
where we have used \eqref{eqn-b-est-1}, \eqref{eqn-con-9} and \eqref{eqn-con-10}. The convergence \eqref{eqn-con-11} implies
	\begin{align}\label{eqn-con-15}
	|\bu_m(\x) |^{r-1}\bu_m(\x)\to |\overline{\bu}(\x) |^{r-1}\overline{\bu}(\x)\ \text{ for a.e. } \ \x\in\mathcal{O}.
\end{align}
Since $\||\bu_m |^{r-1}\bu_m\|_{\L^{\frac{r+1}{r}}}=\|\bu_m\|_{\L^{r+1}}^r\leq C,$ by applying  Lemma \ref{Lem-Lions}, we infer 
\begin{align}\label{eqn-con-16}
	\mathcal{C}(\bu_m)\xrightarrow{w}	\mathcal{C}(\overline{\bu})\ \text{ in }\ \L^{\frac{r+1}{r}}. 
\end{align}
A similar calculation yields 
\begin{align}\label{eqn-con-17}
	\mathcal{C}_0(\bu_m)\xrightarrow{w}	\mathcal{C}_0(\overline{\bu})\ \text{ in }\ \L^{\frac{r+1}{r}}. 
\end{align}
	Taking limit supremum in \eqref{eqn-hemi-approx} and then using \eqref{eqn-con-9}, \eqref{eqn-con-10}, \eqref{eqn-con-14}, \eqref{eqn-con-16} and \eqref{eqn-con-17},  we find 
	\begin{align}\label{eqn-con-18}
		\langle\f,\bv\rangle&\leq 	\mu a(\overline{\bu},\bv)+b(\overline{\bu},\overline{\bu},\bv)+\alpha a_0(\overline{\bu},\bv)+\beta c(\overline{\bu},\bv)+\kappa c_0(\overline{\bu},\bv)+d(\bv,\overline{p})\nonumber\\&\quad+\limsup_{m\to\infty}\int_{\Gamma_1}j^0_m(\bu_{m,\tau};\bv_{\tau})\d S\nonumber\\&\leq 	\mu a(\overline{\bu},\bv)+b(\overline{\bu},\overline{\bu},\bv)+\alpha a_0(\overline{\bu},\bv)+\beta c(\overline{\bu},\bv)+\kappa c_0(\overline{\bu},\bv)+d(\bv,\overline{p})\nonumber\\&\quad+\int_{\Gamma_1}\limsup_{m\to\infty}j^0_m(\bu_{m,\tau};\bv_{\tau})\d S.
	\end{align}
	The convergence \eqref{eqn-con-13} and Hypothesis \ref{hyp-sup-jm-sup} (H5) yield 
	\begin{align}\label{eqn-con-19}
		\limsup_{m\to\infty}j^0_m(\bu_{m,\tau};\bv_{\tau})\leq j^0(\overline{\bu}_{\tau};\bv_{\tau}). 
	\end{align}
	Using \eqref{eqn-con-19} in \eqref{eqn-con-18}, we derive for all $\bv\in\V$
	\begin{align}\label{eqn-con-20}
			\langle\f,\bv\rangle&\leq \mu a(\overline{\bu},\bv)+b(\overline{\bu},\overline{\bu},\bv)+\alpha a_0(\overline{\bu},\bv)+\beta c(\overline{\bu},\bv)+\kappa c_0(\overline{\bu},\bv)+d(\bv,\overline{p})\nonumber\\&\quad+\int_{\Gamma_1} j^0(\overline{\bu}_{\tau};\bv_{\tau})\d S. 
		\end{align}
		Letting $m\to\infty$ in the second equation in \eqref{eqn-hemi-approx}, we get 
		\begin{align}\label{eqn-con-21}
				d(\overline{\bu},q)=0\ \text{ for all }\ q\in Q. 
		\end{align}
	Therefore, from \eqref{eqn-con-19}-\eqref{eqn-con-20}, we conclude that $(\overline{\bu}, \overline{p}) \in \V \times Q$ is a solution of Problem \ref{prob-hemi-1}. By the uniqueness of the solution guaranteed by Theorem \ref{thm-unique}, it follows that $\overline{\bu} = \bu$ and $\overline{p} = p$. Consequently, every subsequence of $\{(\bu_m, p_m)\}_{m \in \mathbb{N}}$ that converges weakly in $\V \times Q$ must have the same limit. Hence, the entire sequence $\{(\bu_m, p_m)\}_{m \in \mathbb{N}}$ converges weakly in $\mathcal{V} \times Q$ to $(\bu, p)$ as $m \to \infty$.
	
		\vskip 0.2cm
	\noindent \textbf{Step 3:} It is only left to prove the strong convergence of the sequence $\{(\bu_m, p_m)\}_{m \in \mathbb{N}}$. Once again, without loss of generality, we may assume that 
$\bu_m\to\bu$ a.e. on $\Gamma_1$ for the sequence of solutions $\{\bu_m\}_{m\in\N}$.  Subtracting \eqref{eqn-hemi-1} from  \eqref{eqn-hemi-approx}, we infer for all $\bv\in\V$ that 
\begin{align}\label{eqn-con-22}
	&	\mu a(\bu_m-\bu,\bv)+b(\bu_m,\bu_m,\bv)-b(\bu,\bu,\bv)+\alpha a_0(\bu_m-\bu,\bv)+\beta [c(\bu_m,\bv)-c(\bu,\bv)]\nonumber\\&\quad+\kappa [c_0(\bu_m,\bv)-c_0(\bu,\bv)]+d(\bv,p_m)-d(\bv,p)+\int_{\Gamma_1}[j^0_m(\bu_{m,\tau};\bv_{\tau})+j^0(\bu_{\tau};-\bv_{\tau})]\d S\nonumber\\&\geq 	\langle\f_m-\f,\bv\rangle. 
\end{align}
For  $\bv=\bu-\bu_m$, using the second equations in \eqref{eqn-hemi-1} and \eqref{eqn-hemi-approx}, we find 
\begin{align}\label{eqn-con-23}
	d(\bu-\bu_m,p_m)-d(\bu-\bu_m,p)=0. 
\end{align}
 Taking $\bv=\bu-\bu_m$ in \eqref{eqn-con-22},  using \eqref{eqn-a-est-2}, \eqref{eqn-b-est-4}  and \eqref{eqn-con-23}, we deduce 
 \begin{align}\label{eqn-con-24}
 	&2\mu\|\bu_m-\bu\|_{\V}^2+\alpha\|\bu_m-\bu\|_{\H}^2+\beta [c(\bu_m,\bu_m-\bu)-c(\bu,\bu_m-\bu)]
 	\nonumber\\&=\mu a(\bu_m-\bu,\bu_m-\bu) +\alpha a_0(\bu_m-\bu,\bu_m-\bu)+\beta [c(\bu_m,\bu_m-\bu)-c(\bu,\bu_m-\bu)]\nonumber\\&\leq \int_{\Gamma_1}[j^0_m(\bu_{m,\tau};\bu_{\tau}-\bu_{m,\tau})+j^0(\bu_{\tau};\bu_{m,\tau}-\bu_{\tau})]\d S\nonumber\\&\quad+b(\bu,\bu,\bu_m-\bu)-b(\bu_m,\bu_m,\bu_m-\bu)\nonumber\\&\quad+ \kappa [c_0(\bu,\bu_m-\bu)-c_0(\bu_m,\bu_m-\bu)]+\langle\f_m-\f,\bu_m-\bu\rangle. 
 \end{align}
 Using \eqref{2.23}, we estimate $\beta [c(\bu_m,\bu_m-\bu)-c(\bu,\bu_m-\bu)]$ as 
 \begin{align}\label{eqn-con-25}
 	\beta [c(\bu_m,\bu_m-\bu)-c(\bu,\bu_m-\bu)]&\geq \frac{\beta}{2}\||\bu_m|^{\frac{r-1}{2}}(\bu_m-\bu)\|_{\H}^2+\frac{\beta}{2}\||\bu|^{\frac{r-1}{2}}(\bu_m-\bu)\|_{\H}^2\nonumber\\&\geq \frac{\beta}{2^{r-1}}\|\bu_m-\bu\|_{\L^{r+1}}^{r+1}.
 \end{align}
 An application of  Taylor's formula (\cite[Theorem 7.9.1]{PGC}) yields 
 \begin{align}\label{eqn-con-26}
 &	|	\kappa||\langle\mathcal{C}_0(\bu_m)-\mathcal{C}_0(\bu),\bu_m-\bu\rangle|\nonumber\\&= |\kappa|\bigg|\bigg<\int_0^1\mathcal{C}_0^{\prime}(\theta\bu_m+(1-\theta)\bu)\d\theta(\bu_m-\bu),(\bu_m-\bu)\bigg>\bigg|\nonumber\\&\leq |\kappa| q2^{q-1}\bigg<\int_0^1|\theta\bu_m+(1-\theta)\bu|^{q-1}\d\theta|\bu_m-\bu|,|\bu_m-\bu|\bigg>\nonumber\\&\leq |\kappa|q2^{q-1}\left<\left(|\bu_m|^{q-1}+|\bu|^{q-1}\right)|\bu_m-\bu|,|\bu_m-\bu|\right>\nonumber\\&=   |\kappa| q2^{q-1}\||\bu_m|^{\frac{q-1}{2}}(\bu_m-\bu)\|_{\H}^2+ |\kappa| q2^{q-1}\||\bu|^{\frac{q-1}{2}}(\bu_m-\bu)\|_{\H}^2. 
 \end{align}
 Using H\"older's inequality, we estimate the term $ |\kappa| q2^{q-1}\||\bu_m|^{\frac{q-1}{2}}(\bu_m-\bu)\|_{\H}^2$ as 
 \begin{align}\label{eqn-con-27}
 	& |\kappa| q2^{q-1}\||\bu_m|^{\frac{q-1}{2}}(\bu_m-\bu)\|_{\H}^2\nonumber\\&= |\kappa| q2^{q-1}\int_{\mathcal{O}}|\bu_m(\x)|^{q-1}|\bu_m(\x)-\bu(\x)|^2\d\x \nonumber\\&= |\kappa| q2^{q-1}\int_{\mathcal{O}}|\bu_m(\x)|^{q-1}|\bu_m(\x)-\bu(\x)|^{\frac{2(q-1)}{r-1}}|\bu_m(\x)-\bu(\x)|^{\frac{2(r-q)}{r-1}}\d\x \nonumber\\&\leq |\kappa| q2^{q-1}\bigg(\int_{\mathcal{O}}|\bu_m(\x)|^{r-1}|\bu_m(\x)-\bu(\x)|^2\d\x\bigg)^{\frac{q-1}{r-1}}\bigg(\int_{\mathcal{O}}|\bu_m(\x)-\bu(\x)|^2\d\x\bigg)^{\frac{r-q}{r-1}}\nonumber\\&\leq\frac{\beta}{4}\int_{\mathcal{O}}|\bu_m(\x)|^{r-1}|\bu_m(\x)-\bu(\x)|^2\d\x+\varrho_{1,r}\int_{\mathcal{O}}|\bu_m(\x)-\bu(\x)|^2\d\x,
 \end{align}
 where 
 \begin{align}\label{eqn-con-28}
 	\varrho_{1,r}=\left(\frac{r-q}{r-1}\right)\left(\frac{4(q-1)}{\beta(r-1)}\right)^{\frac{q-1}{r-q}}\left( |\kappa| q2^{q-1}\right)^{\frac{r-1}{r-q}}. 
 \end{align}
 A similar calculation yields 
 \begin{align}\label{eqn-con-29}
 	& |\kappa| q2^{q-1}\||\bu|^{\frac{q-1}{2}}(\bu_m-\bu)\|_{\H}^2\nonumber\\&\leq\frac{\beta}{4}\int_{\mathcal{O}}|\bu(\x)|^{r-1}|\bu_m(\x)-\bu(\x)|^2\d\x+\varrho_{1,r}\int_{\mathcal{O}}|\bu_m(\x)-\bu(\x)|^2\d\x.
 \end{align}
 Using \eqref{eqn-con-27} and \eqref{eqn-con-29} in \eqref{eqn-con-26}, we deduce 
 \begin{align}\label{eqn-con-31}
 	|\kappa||\langle\mathcal{C}_0(\bu_m)-\mathcal{C}_0(\bu),\bu_m-\bu\rangle|&\leq\frac{\beta}{4}\||\bu_m|^{\frac{r-1}{2}}(\bu_m-\bu)\|_{\L^2}^2+\frac{\beta}{4}\||\bu|^{\frac{r-1}{2}}(\bu_m-\bu)\|_{\L^2}^2\nonumber\\&\quad+2\varrho_{1,r}\|\bu_m-\bu\|_{\H}^2. 
 \end{align}
 Let us now consider the term $|b(\bu,\bu,\bu_m-\bu)-b(\bu_m,\bu_m,\bu_m-\bu)|$ and estimate it using \eqref{eqn-b-est-3}, \eqref{eqn-b-est-4},  H\"older's, Korn's, Gagliardo-Nirenberg's and  Young's inequalities as 
 \begin{align}\label{eqn-con-32}
 	&|b(\bu,\bu,\bu_m-\bu)-b(\bu_m,\bu_m,\bu_m-\bu)|\nonumber\\&\leq  |b(\bu-\bu_m,\bu,\bu_m-\bu)|+|b(\bu_m,\bu-\bu_m,\bu_m-\bu)|\nonumber\\&= |b(\bu-\bu_m,\bu_m-\bu,\bu)|\nonumber\\&\leq   \|\nabla(\bu_m-\bu)\|_{\H}\|\bu_m-\bu\|_{\L^4}\|\bu\|_{\L^4}
 	\nonumber\\&\leq C_gC_k\|\bu_m-\bu\|_{\V}\|\bu_m-\bu\|_{\H}^{1-\frac{d}{4}}\|\bu_m-\bu\|_{\H^1}^{\frac{d}{4}}\|\bu\|_{\L^4}
 	\nonumber\\&\leq C_gC_k\|\bu_m-\bu\|_{\H}^{1-\frac{d}{4}}\|\bu_m-\bu\|_{\V}^{1+\frac{d}{4}}\|\bu\|_{\L^4}\nonumber\\&\leq \frac{\mu}{2}\|\bu_m-\bu\|_{\V}^2+\varrho_2\|\bu\|_{\L^4}^{\frac{8}{4-d}}\|\bu_m-\bu\|_{\H}^2,
 \end{align}
 	where 
 \begin{align}\label{eqn-con-33}
 	\varrho_2=C_g^{\frac{8}{4-d}}C_k^{\frac{8}{4-d}}\left(\frac{8}{4-d}\right)\left(\frac{4+d}{4\mu}\right)^{\frac{4+d}{4-d}}.
 \end{align}
 The Cauchy-Schwarz and Young's inequalities yield 
 \begin{align}\label{eqn-con-34} 
 	|\langle\f_m-\f,\bu_m-\bu\rangle|\leq\|\f_m-\f\|_{\V^*}\|\bu_m-\bu\|_{\V}\leq\frac{\mu}{2}\|\bu_m-\bu\|_{\V}^2+\frac{1}{2\mu}\|\f_m-\f\|_{\V^*}^2. 
 \end{align}
 Combining \eqref{eqn-con-25}, \eqref{eqn-con-31}, \eqref{eqn-con-32} and \eqref{eqn-con-34}, and substituting the resultant  in \eqref{eqn-con-24}, we deduce 
 \begin{align}\label{eqn-con-35}
 &	\mu\|\bu_m-\bu\|_{\V}^2+\alpha\|\bu_m-\bu\|_{\H}^2+\frac{\beta}{2^r}\|\bu_m-\bu\|_{\L^{r+1}}^{r+1}\nonumber\\&\leq \int_{\Gamma_1}[j^0_m(\bu_{m,\tau};\bu_{\tau}-\bu_{m,\tau})+j^0(\bu_{\tau};\bu_{m,\tau}-\bu_{\tau})]\d S\nonumber\\&\quad+\left(2\varrho_{1,r}+\varrho_2\|\bu\|_{\L^4}^{\frac{8}{4-d}}\right)\|\bu_m-\bu\|_{\H}^2+\frac{1}{2\mu}\|\f_m-\f\|_{\V^*}^2. 
 \end{align}
 Let us now apply Hypotheses \ref{hyp-sup-jm-sup} and \ref{hyp-sup-j}, and Proposition \ref{prop-sub-diff} (i) to find 
 \begin{align}\label{eqn-con-36}
 	\limsup_{m\to\infty} \int_{\Gamma_1}j^0_m(\bu_{m,\tau};\bu_{\tau}-\bu_{m,\tau})\d S&\leq  	\int_{\Gamma_1}\limsup_{m\to\infty} j^0_m(\bu_{m,\tau};\bu_{\tau}-\bu_{m,\tau})\d S\nonumber\\&\leq \int_{\Gamma_1}j^0(\bu_{\tau};\bu_{\tau}-\bu_{\tau})\d S=\int_{\Gamma_1}j^0(\bu_{\tau};\boldsymbol{0})\d S=0. 
 \end{align}
 Similarly, applying Hypothesis  \ref{hyp-sup-j} and Proposition \ref{prop-sub-diff} (ii), we get 
  \begin{align}\label{eqn-con-37}
 	\limsup_{m\to\infty} \int_{\Gamma_1}j^0(\bu_{m,\tau};\bu_{\tau}-\bu_{m,\tau})\d S&\leq  	\int_{\Gamma_1}\limsup_{m\to\infty} j^0(\bu_{m,\tau};\bu_{\tau}-\bu_{m,\tau})\d S\nonumber\\&\leq \int_{\Gamma_1}j^0(\bu_{\tau};\bu_{\tau}-\bu_{\tau})\d S=\int_{\Gamma_1}j^0(\bu_{\tau};\boldsymbol{0})\d S=0. 
 \end{align}
 Since $\f_m\to\f$ in $\V^*$ and $\bu_m\to\bu$ in $\H$ (see \eqref{eqn-con-10}), taking limit supremum on both sides of \eqref{eqn-con-35}, we immediately have 
 \begin{align*}
 	\limsup_{m\to\infty}\|\bu_m-\bu\|_{\V}\leq 0,
 	\end{align*}
 	so that 
 	 \begin{align*}
 		\lim_{m\to\infty}\|\bu_m-\bu\|_{\V}\leq 0,
 	\end{align*}
 	that is, $\bu_m\to\bu$ in $\V$. Passing limit as $m\to\infty$ in \eqref{eqn-con-7}, we deduce $p_m\to p$ in $Q$, which completes the proof.  
	\end{proof}
	If we consider perturbations in the external force density only, then we obtain a similar result for the following problem: 
		\begin{problem}\label{prob-hemi-approx-1}
		Find $\bu_m\in\V$ and $p_m\in Q$  such that
		\begin{equation}\label{eqn-hemi-approx-1}
			\left\{
			\begin{aligned}
				\mu a(\bu_m,\bv)+b(\bu_m,\bu_m,\bv)+\alpha a_0(\bu_m,\bv)&+\beta c(\bu_m,\bv)+\kappa c_0(\bu_m,\bv)+d(\bv,p_m)\\+\int_{\Gamma_1}j^0(\bu_{m,\tau};\bv_{\tau})\d S&\geq 	\langle\f_m,\bv\rangle, \ \text{ for all }\ \bv\in \V,\\
				d(\bu_m,q)&=0,\ \text{ for all }\ q\in Q. 
			\end{aligned}\right.
		\end{equation}
	\end{problem}
	
	\begin{corollary}\label{cor-conver}
		Let Hypotheses \ref{hyp-sup-j}  be satisfied. Then under assumptions of Theorem \ref{thm-unique}, Problem \ref{prob-hemi-1} has a unique solution $(\bu, p)\in\V\times Q$, Problem \ref{prob-hemi-approx} has a unique solution $(\bu_m,p_m) \in\V\times Q$, and
		\begin{align}\label{eqn-bound-discrete-1}
			\|\bu\|_{\V}^2+	\|\bu\|_{\L^{r+1}}^{r+1}\leq \widetilde{K} \ \text{ and }\ \|\bu_m\|_{\V}^2+	\|\bu_m\|_{\L^{r+1}}^{r+1}\leq \widetilde{K},
		\end{align}
		where $\widetilde{K}=2\max\left\{\frac{1}{\left(2\mu-k_1\lambda_0^{-1}\right)},\frac{1}{\beta}\right\} K$ and $K$ is defined in \eqref{eqn-value-k}. 
		
		Furthermore, if  $\|\f_m-\f\|_{\V^*}\to 0$, then 
		\begin{align*}
			\bu_m\to \bu \ \text{ in }\ \V\ \text{ and }\ p_m\to p\ \text{ in }\ Q. 
		\end{align*}
	\end{corollary}

	\section{An optimal control problem}\setcounter{equation}{0}\label{sec5} We consider the optimal control problem for the CBFeD hemivariational inequality \eqref{eqn-hemi-1}, where the external force density $\f$ taken as the control variable, belongs to the control space $\V^*$. Let \( \V_{\mathrm{ad}} \subset \V^*_{m_0} \) be the set of admissible controls, and let \( \mathcal{R} : \V \times Q \times \V^* \to \mathbb{R} \cup \{ +\infty \} \) be the objective functional, which has the following form:
	\begin{align*}
		\mathcal{R}(\f)=R(\bu(\f),p(\f),\f), 
	\end{align*}
	where \((\bu(\f), p(\f)) \in \V \times Q\) denotes the solution of Problem \ref{prob-hemi-1} corresponding to the control \(\f\). For notational simplicity, we represent the cost function as
	\begin{align*}
			\mathcal{R}(\f)=R(\bu,p,\f), 
	\end{align*}
	where $(\bu,p)=(\bu(\f),p(\f))$. The optimal control problem can then be formulated as
	\begin{align}\label{eqn-opt-1}
		\inf\left\{\mathcal{R}(\f):\f\in\V^*_{\mathrm{ad}}\right\}. 
	\end{align}
	Regarding problem \eqref{eqn-opt-1}, we impose the following assumptions on the control space and objective functional:
	\begin{hypothesis}\label{hyp-opt} 
		\begin{itemize}
		\item [$(H(\V_{\mathrm{ad}}^*))$] The set  \( \V_{\mathrm{ad}}^* \subset \V^*_{m_0} \) is a nonempty and compact subset of $\V^*$. 
		\item [$(H(R))$] The mapping $R : \V \times Q\times\V^*\to\R\cup\{+\infty\}$ is lower semicontinuous, that is, if $\bu_m\to\bu$ in $\V$, $p_n\to p$ in $Q$ and $\f_n\to\f$ in $\V^*$, then 
		\begin{align*}
			R(\bu,p,\f)\leq\liminf_{m\to\infty}R(\bu_m,p_m,\f_m). 
		\end{align*}
		\end{itemize}
	\end{hypothesis}

	We now establish the existence of solutions to the optimal control problem \eqref{eqn-opt-1}. 
	\begin{theorem}\label{thm-optimal-exis}
		Let Hypotheses \ref{hyp-sup-j} and \ref{hyp-opt} hold. Under the assumptions of Theorem \ref{thm-unique}, the optimal control problem \eqref{eqn-opt-1} has a solution.
	\end{theorem}
	\begin{proof}
		Let $\{\f_m\}_{m\in\N}$ be a minimizing sequence for problem \eqref{eqn-opt-1}, that is, each 
		$\f_m\in\V_{\mathrm{ad}}^*$ and 
		\begin{align}\label{eqn-min-1}
			\lim\limits_{m\to\infty}\mathcal{R}(\f_m)=	\inf\left\{\mathcal{R}(\f):\f\in\V^*_{\mathrm{ad}}\right\}=:\mathcal{R}. 
		\end{align}
		Let \((\bu_m, p_m) \in \V \times Q\) be the unique solution to Problem \ref{prob-hemi-approx-1}. Since the sequence \(\{\f_m\}_{m\in\N}\) lies in \(\V^*_{\mathrm{ad}}\), a compact subset of \(\V^*\), there exists a subsequence (still denoted by \(\{\f_m\}_{m\in\N}\)) that converges strongly in \(\V^*\) to some \(\f^* \in \V^*_{\mathrm{ad}}\).  
			By Corollary \ref{cor-conver}, the corresponding solutions \((\bu_m, p_m)\) converge strongly to \((\bu^*, p^*)\) in \(\V \times Q\), where \((\bu^*, p^*)\) solves Problem \ref{prob-hemi-1} with \(\f =\f^*\). Under  Hypothesis \ref{hyp-opt} $(H(R))$, it follows that  
				\[  
		\mathcal{R} \leq \mathcal{S}(\bu^*, p^*, \f^*) \leq \liminf_{m\to \infty} \mathcal{S}(\bu_m, p_m, \f_m) = \mathcal{R}.  
		\]  
			This implies that \(\f^*\) is indeed an optimal control, solving problem  \eqref{eqn-opt-1}. 
	\end{proof}
	
	\begin{remark}
			Typical examples of the objective functional are given by:
		\begin{align}\label{eqn-cost-1}
			\mathcal{R}_1(\f)=\frac{\alpha_1}{2}\int_{\mathcal{O}}\|\bu(\x)-\bu_d(\x)\|^2_{\R^d}\d\x+\frac{\alpha_2}{2}\int_{\mathcal{O}}|p(\x)-p_d(\x)|^2\d\x+\frac{\alpha_3}{2}\int_{\mathcal{O}}\|\f(\x)\|_{\R^d}^2\d\x
		\end{align}
		or 
		\begin{align}\label{eqn-cost-2}
			\mathcal{R}_2(\f)=\frac{\alpha_1}{2}\int_{\mathcal{O}}\|\mathrm{curl \ }\bu(\x)\|^2_{\R^d}\d\x+\frac{\alpha_2}{2}\int_{\mathcal{O}}|p(\x)-p_d(\x)|^2\d\x+\frac{\alpha_3}{2}\int_{\mathcal{O}}\|\f(\x)\|_{\R^d}^2\d\x,
		\end{align}
		where the external force density $\f\in\H\subset\V^*$  is the control, $(\bu, p)$ is the unique solution of Problem \ref{prob-hemi-1} corresponding to $\f$.  
		The target states $\bu_d$ and $p_d$ represent the desired velocity and pressure distributions, respectively. The weighting coefficients $\alpha_1,\alpha_2\geq 0$  are fixed parameters satisfying $\alpha_1+\alpha_2=1$, with $\alpha_3>0$ serving as a regularization parameter. These constants govern the relative contributions of the three terms in the objective functional $\mathcal{R}(\f)$. When $\alpha_2=0$ in $\mathcal{R}_1(\f)$, the optimization problem reduces to pure velocity tracking, where the control acts exclusively to match the target velocity field. The regularization term ($\alpha_3>0$) imposes an energy constraint on the control forces, ensuring physically realizable solutions by maintaining finite actuation magnitudes. The control aims to determine an optimal force density that minimizes deviations between the actual and target velocity-pressure fields $(\bu,p)$ and $(\bu_d,p_d)$. When $\alpha_2=0$ in $\mathcal{R}_2(\f)$,  the term $\mathcal{R}_2(\f)$ acts as an enstrophy-minimizing functional, employing the quadratic vorticity penalty $\|\mathrm{curl \ }\bu\|_{\H}^2$ to regulate rotational kinetic energy. Hypothesis  \ref{hyp-opt} $(H(R))$ is verified by the specific choices  $\mathcal{R}_1(\f)$ and $\mathcal{R}_2(\f)$  of the objective functional. 
		\end{remark}
		
		\begin{remark}
		Hypothesis \ref{hyp-opt} ($H(\V_{\mathrm{ad}}^*)$) can be relaxed to the condition that $\V_{\mathrm{ad}}^* $ is nonempty when using the cost functionals from \eqref{eqn-cost-1} or \eqref{eqn-cost-2}. Here we are not even assuming $\V_{\mathrm{ad}}^* \subset\V_{m_0}^*$. Under the condition \eqref{eqn-min-1}, we obtain the estimates
		\begin{align}
			\mathcal{R} \leq \mathcal{R}(\f_m) \leq \mathcal{R} + \frac{1}{m},
		\end{align}
		along with the uniform bound
		\begin{align}
			\|\f_m\|_{\H} \leq C
		\end{align}
		for some constant $C>0$. Applying the Banach-Alaoglu theorem, we extract a subsequence of $\{\f_m\}_{m\in\N}$ (still denoted by the same symbol) that converges weakly to $\f$ in $\H$.
		
		In the proof of Theorem \ref{thm-conv} (Step 3), the convergence $\langle\f_m-\f,\bu_m-\bu\rangle \to 0$ follows directly from the strong convergence $\bu_m\to\bu$ in $\H$ (see \eqref{eqn-con-10}) and the fact that $\|\f_m\|_{\H} \leq C$, as shown by the equality
		\begin{align}
			\langle\f_m-\f,\bu_m-\bu\rangle &= (\f_m-\f,\bu_m-\bu)\nonumber\\&\leq\|\f_m-\f\|_{\H}\|\bu_m-\bu\|_{\H}\nonumber\\&\leq (\|\f_m\|_{\H}+\|\f\|_{\H})\|\bu_m-\bu\|_{\H} \to 0.
		\end{align}
		Therefore the convergence $\bu_m\to\bu$ in $\V$ also follows (see \eqref{eqn-con-35}).
	\end{remark}

	\section{Numerical approximation of the optimal control problem}\setcounter{equation}{0} \label{sec6}
	In this section, we investigate the numerical approximation of the optimal control problem \ref{eqn-opt-1}. We maintain all the assumptions given in Theorem \ref{thm-optimal-exis}.  Let \( h > 0 \) be a discretization parameter, and consider the finite-dimensional approximations
		\(
	\left\{\V_h, Q_h, \V^{*}_{\mathrm{ad},h}\right\}_{h}
	\)
		of the spaces \( \left( \V, Q, \V^{*}_{\mathrm{ad}} \right) \) as \( h \to 0 \).  Inspired from \cite[Section 4]{WWXCWH}, we impose the following assumptions on the discrete approximation spaces.
		
		\begin{hypothesis}\label{hyp-discrete}
			\begin{enumerate}
				\item [$(H(\V_h))$] The discrete subspace $\V_h \subset\V$ has the approximation property that any $\bv\in\V$ can be approximated by elements $\bv_h \in\V_h$ with $\|\bv_h-\bv\|_{\V}\to 0$ as $h\to 0$.
				\item [$(H(Q_h))$]   The discrete subspace $Q_h \subset Q$ has the approximation property that any $q\in Q$ can be approximated by elements $q_h \in Q_h$ with $\|q_h-q\|_{Q}\to 0$ as $h\to 0$.
				\item [$(H(\V_{\mathrm{ad},h}^*))$]  The discrete subspace \( \V_{\mathrm{ad},h}^* \subset \V^*_{m_0} \) is a nonempty and compact subset of $\V^*$. For any $\f\in\V_{\mathrm{ad}}^*$, there exists $\f_h \in \V_{\mathrm{ad},h}^*$ with $\|\f_h-\f\|_{\V^*}\to 0$ as $h\to 0$  and for any sequence $\{\f_h\}_h$, $\f_h\in \V_{\mathrm{ad},h}^*$, there exists a sequence $\{\f_h'\}_h\subset \V_{\mathrm{ad},h}^*$ with $\|\f_h-\f_h'\|_{\V^*}\to 0$ as $h\to 0$. 
			\end{enumerate}
		\end{hypothesis}
		Note that Hypothesis \ref{hyp-discrete}  $(H(\V_{\mathrm{ad},h}^*))$  is weaker than assuming that $\V_{\mathrm{ad},h}^*\subset\V_{\mathrm{ad}}^*$. Moreover, we have the following assumption on the discrete inf-sup condition. 
		\begin{hypothesis}\label{hyp-discrete-inf-sup}
			There exists a constant $\vartheta_1>0$ such that 
			\begin{align}\label{eqn-numer-1}
				\vartheta_1\|q_h\|_Q\leq \sup_{\bv\in\V_{0,h}}\frac{d(\bv_h,q_h)}{\|\bv_h\|_{\V}}\ \text{ for all }\ q_h\in Q_h,
			\end{align}
			where $\V_{0,h}:=\V_0\cap\V_h$. 
		\end{hypothesis}
		We remark that the condition \ref{eqn-numer-1}  is commonly referred to as the \emph{Babu\u{s}ka-Brezzi condition} in the literature.
		We now catalog commonly used finite element spaces that meet the Babu\u{s}ka-Brezzi requirements.
		Consider a polygonal domain (for \( d = 2 \)) or a polyhedral domain (for \( d = 3 \)). Let \( \{\mathcal{T}_h\}_h \) be a regular family of finite element partitions of \( \mathcal{O} \) into triangular (2D) or tetrahedral (3D) elements.  For an integer \( k \geq 0 \), let 
		 \( P_k(T) \) denote the space of polynomials of total degree \( \leq k \) on an element \( T \) and 
	 \( B(T) \) represents the space of bubble functions on \( T \).  
			Among the well-known stable finite element spaces, one classical choice is the \emph{Mini element} (\cite[Section 2]{DNAFB}), 
			\begin{align*}
				\V_h&=\left\{\bv_h\in\V_h\cap\mathrm{C}^0(\overline{\mathcal{O}})^d: \bv_h\big|_T\in[P_1(T)]^d\oplus[B(T)]^d\ \text{ for all }\ T\in\mathcal{T}_h\right\},\\
				Q_h&=\left\{q_h\in Q_h\cap\mathrm{C}^0(\overline{\mathcal{O}}): q_h\big|_T\in P_1(T)\ \text{ for all } \ T\in\mathcal{T}_h\right\},
			\end{align*}
			or $P_2/P_1 $ finite element pair (\cite[Section IV.4.2]{VGPR})
			\begin{align*}
			\V_h&=\left\{\bv_h\in\V_h\cap\mathrm{C}^0(\overline{\mathcal{O}})^d: \bv_h\big|_T\in[P_1(T)]^d\ \text{ for all }\ T\in\mathcal{T}_h\right\},\\
			Q_h&=\left\{q_h\in Q_h\cap\mathrm{C}^0(\overline{\mathcal{O}}): q_h\big|_T\in P_1(T)\ \text{ for all } \ T\in\mathcal{T}_h\right\},
				\end{align*}
			which satisfies the Babu\u{s}ka-Brezzi condition.  
			
			The verification of commonly used finite element spaces $(H(\V_h))$ and $(H(Q_h))$ follows a standard technique in finite element analysis, combining two key ingredients: 
			\begin{enumerate} \item [(1)] the density of smooth functions in the relevant function spaces ($\V$ or $Q$), where we note that $\V \cap \mathrm{C}^{\infty}(\overline{\mathcal{O}};\R^d)$ is dense in $\V$ (see \cite[Section 7.1]{WHMS2} and references therein) and the density of $\mathrm{C}^{\infty}(\overline{\mathcal{O}})$ in $\mathrm{L}^2(\mathcal{O})$ is a classical result from Sobolev space theory, which immediately implies the density of  $Q\cap \mathrm{C}^{\infty}(\overline{\mathcal{O}})$ in $Q$; and \item [(2)] standard error estimates for finite element interpolations of smooth functions, as found in fundamental finite element textbooks (\cite{}) and exemplified by the proof of \cite[Theorem 10.4.1]{AKWH}. 
			\end{enumerate}

			When approximating the control space, we note that the admissible set $\V_{\mathrm{ad}}^*$ lacks a vector space structure, meaning  $\V_{\mathrm{ad},h}^*$ is generally not contained within $\V_{\mathrm{ad}}^*$. Following established approaches for obstacle problems, we adopt assumption Hypothesis \ref{hyp-discrete} $(H(\V_{\mathrm{ad},h}^*))$. The specific approximation space for  $(H(\V_{\mathrm{ad},h}^*))$ depends on $\V_{\mathrm{ad}}^*$'s formulation, with the final property automatically satisfied if $\V_{\mathrm{ad},h}^*\subset \V_{\mathrm{ad}}^*$ for all $h\in(0,1)$. As discussed in \cite[Section 4]{WWXCWH}, to demonstrate this assumption's validity, consider an example where $\g_1, \g_2\in\L^{\infty}(\mathcal{O})$ define $ G(\x) :=\max\{|\g_1(\x)|,|\g_2(\x)|\}$ for a.e. $\x\in\mathcal{O}$, with $\|G\|_{\V^*}\leq C\|G\|_{\mathrm{L}^{\infty}}\leq m_0$. The admissible control set is then given by $$\V_{\mathrm{ad}}^* = \{ \f\in\H : \g_1\leq\f\leq\g_2\ \text{ a.e. in }\ \mathcal{O}\},$$ illustrating how these constraints manifest in practice.
			
			Based on the given definition and properties of \( \g_1 \) and \( \g_2 \), we conclude that \( \V_{\mathrm{ad}}^*\subset \V^*_{m_0} \). Since \( \V_{\mathrm{ad}}^* \) forms a closed, nonempty, convex subset of \( \H \), it is consequently nonempty and weakly closed in \( \H \) (see \cite[Section 3.3]{AKWH}). Furthermore, due to the compact embedding of \( \H \) into \( \V^* \),  Hypothesis \ref{hyp-discrete} $(H(\V_{\mathrm{ad},h}^*))$ is satisfied.  
			
			Let us define 
			\begin{align}
				\boldsymbol{U}_h=\left\{\f_h\in\H\cap\mathrm{C}^0(\overline{\mathcal{O}}):\f_h\big|_T\in P_1(T)\ \text{ for all }\ T\in \mathcal{T}_h\right\}.
			\end{align}
			Let $\mathcal{N}_h$ denote the set of nodes for the finite element spaces $\boldsymbol{U}_h$, and define
			\begin{align}
				\V_{\mathrm{ad},h}^*=\left\{\f_h\in\boldsymbol{U}_h: \text{ for all }\ \boldsymbol{b}\in\mathcal{N}_h,\  \g_1(\boldsymbol{b})\leq \f(\boldsymbol{b})\leq \g_2(\boldsymbol{b})\right\}. 
			\end{align}
			Consequently, $\V_{\mathrm{ad},h}^*\subset \V_{m_0}^*$ forms a nonempty, compact subset of $\V$, and by \cite[Theorem 5.1.2]{PGC1}, Hypothesis \ref{hyp-discrete} $(H(\V_{\mathrm{ad},h}^*))$ is satisfied.  We now present the discrete formulation of Problem \ref{prob-hemi-1}. 
			\begin{problem}\label{prob-hemi-discrete}
				Given $\f_h\in	\V_{\mathrm{ad},h}^*$, find $\bu_h\in\V_h$ and $p_h\in Q_h$  such that
				\begin{equation}\label{eqn-hemi-discrete}
					\left\{
					\begin{aligned}
						\mu a(\bu_h,\bv_h)+b(\bu_h,\bu_h,\bv_h)+\alpha a_0(\bu_h,\bv_h)&+\beta c(\bu_h,\bv_h)+\kappa c_0(\bu_h,\bv_h)+d(\bv_h,p_h)\\+\int_{\Gamma_1}j^0(\bu_{h,\tau};\bv_{h,\tau})\d S&\geq 	\langle\f_h,\bv_h\rangle, \ \text{ for all }\ \bv_h\in \V_h,\\
						d(\bu_h,q_h)&=0,\ \text{ for all }\ q_h\in Q_h. 
					\end{aligned}\right.
				\end{equation}
			\end{problem}
			An existence, uniqueness and boundedness result holds for Problem \ref{eqn-hemi-discrete}, analogous to Theorem \ref{thm-unique}. This establishes the foundation for developing numerical approximations of the optimal control formulation \eqref{eqn-opt-1}, which is given by 
			\begin{align}\label{eqn-opt-2}
				\inf\left\{\mathcal{R}_h(\f_h):\f_h\in	\V_{\mathrm{ad},h}^* \right\},
			\end{align}
			where 
			\begin{align}
				\mathcal{R}_h(\f_h)=R(\bu_h,p_h,\f_h),
			\end{align}
			and $(\bu_h,p_h) = (\bu_h (\f_h) , p_h(\f_h))$ is the solution of Problem \ref{prob-hemi-discrete}. 
			
			Following the framework of the optimal control problem \eqref{eqn-opt-1}, the discrete formulation \eqref{eqn-opt-2}  admits a solution under the same set of assumptions: Hypotheses \ref{hyp-sup-j}, \ref{hyp-opt} $(H(R))$, \eqref{eqn-unique-con-1} or \eqref{eqn-unique-con-2}, along with the discrete stability conditions given in Hypotheses \ref{hyp-discrete} and \ref{hyp-discrete-inf-sup}.
				To establish the convergence properties of the discrete solution, we first demonstrate the continuous dependence result for Problem \ref{prob-hemi-discrete} with respect to control perturbations, specifically showing stability as $\f_h\to\f$ in $\V^*$ as $h\to 0$. 
				
				\begin{theorem}\label{thm-discrete-conv}
				Let 	Hypotheses \ref{hyp-sup-j}, \ref{hyp-opt} $(H(R))$, \eqref{eqn-unique-con-1} or \eqref{eqn-unique-con-2} and Hypotheses \ref{hyp-discrete}  be satisfied. Let $(\bu,p)\in\V\times Q$ and $(\bu_h, p_h)\in\V_h\times Q_h$ be the solutions of Problem \ref{prob-hemi-1} and Problem \ref{prob-hemi-discrete}, respectively.If $\f_h\to\f$ in $\V^*$, then 
				\begin{align}
					\|\bu_h-\bu\|_{\V}+\|p_h-p\|_Q\to 0\ \text{ as } \ h\to 0. 
					\end{align}
				\end{theorem}
				\begin{proof}
					We establish the proof through several steps.
					\vskip 0.2cm
					\noindent\textbf{Step 1:} Let us first establish uniform boundedness of the discrete solutions. Since  $\V_{\mathrm{ad},h}^*\subset \V_{m_0}^*$, the sequence of discrete controls $\{\f_h\}$ is bounded in $\V^*$. Consequently, by the discrete counterpart of Proposition \ref{prop-energy-est}, the corresponding solutions 	$\{\bu_h\}$  remain uniformly bounded in $\V\cap\L^{r+1}$. Since $1\leq r\leq \frac{2d}{d-2}$, by Sobolev's embedding, we infer $\V\subset\L^{r+1}$ and hence $\V\cap\L^{r+1}=\V$. Similar to \eqref{eqn-con-1}, we have  for all $\bv_h\in\V_{h,0}$ that 
					\begin{align}\label{eqn-numer-0}
							\mu a(\bu_h,\bv_h)+b(\bu_h,\bu_h,\bv_h)&+\alpha a_0(\bu_h,\bv_h)+\beta c(\bu_h,\bv_h)+\kappa c_0(\bu_h,\bv_h)\nonumber\\+d(\bv_h,p_h)&=\langle\f_h,\bv_h\rangle. 
					\end{align}
					By the Babu\u{s}ka-Brezzi condition \eqref{eqn-numer-1}, we find 
					\begin{align}
						\vartheta_1\|p_h\|_Q&\leq \sup_{\bv\in\V_{0,h}}\frac{d(\bv_h,p_h)}{\|\bv_h\|_{\V}}\nonumber\\& =  \sup_{\bv\in\V_{0,h}}\frac{\begin{array}{l}\langle\f_h,\bv_h\rangle-[	\mu a(\bu_h,\bv_h)+b(\bu_h,\bu_h,\bv_h)+\alpha a_0(\bu_h,\bv_h)\\ \quad+\beta c(\bu_h,\bv_h)+\kappa c_0(\bu_h,\bv_h)]\end{array}}{\|\bv_h\|_{\V}}\nonumber\\&\leq \left(\|\f_h\|_{\V^*}+(2\mu+\alpha C_k^2)\|\bu_h\|_{\V}+C_b\|\bu_h\|_{\V}^2+[\beta C_s+|\kappa|C_s|\mathcal{O}|^{\frac{r-q}{r+1}}]\|\bu_h\|_{\L^{r+1}}^r \right),
					\end{align}
					where we have used the estimates \eqref{eqn-equiv-1}, \eqref{eqn-a-est-1}, \eqref{eqn-b-est-1} and \eqref{eqn-c-est}.  Since the sequences $\{\|\bu_h\|_{\V}\}_h$ and $\{\|\f_h\|_{\V^*}\}_h$ are uniformly  bounded, $\{\|p_h\|_Q\}_h$ is also bounded. 
					
					\vskip 0.2cm
					\noindent\textbf{Step 2:} Since the sequences $\{\|\bu_h\|_{\V}\}_h$ and $\{\|p_h\|_Q\}_h$ are bounded, an application of the Banach-Alaoglu theorem guarantees the existence of subsequences of $\{\|\bu_h\|_{\V}\}_h$ and $\{\|p_h\|_Q\}_h$ (not relabeled)  and elements $\overline{\bu}\in\V$ and $\overline{p}\in Q$ such that 
					\begin{align}\label{eqn-nume-1}
						\bu_h\xrightarrow{w}\overline{\bu}\ \text{ in }\ \V\ \text{ and }\ p_h\xrightarrow{w}\overline{p}\ \text{ in }\ Q \ \text{ as }\ h\to 0. 
					\end{align}
					Since the embedding $\V\subset\H$ is compact, we have the following strong convergence:
					\begin{align}\label{eqn-nume-1-1}
						\bu_h\to \overline{\bu}\ \text{ in }\ \H. 
					\end{align}
					Along a further subsequence (still denoted by the same symbol), we obtain 
					\begin{align}
						\bu_h(\x)\to\overline{\bu}(\x)\ \text{ for a.e. }\ \x\in\H. 
					\end{align}
					
					\vskip 0.2cm
					\noindent\textbf{Step 3:} Our next aim is to show that 
					\begin{align}\label{eqn-nume-2}
						d(\overline{\bu},q)=0\ \text{ for all }\ q\in Q. 
					\end{align}
					For an arbitrarily fixed $q\in Q$, from Hypothesis \ref{hyp-discrete} $(H(Q_h))$, we infer the existence of a $q_h\in Q_h$ such that 
					\begin{align}\label{eqn-nume-3}
						q_h\to q\ \text{ in } \ Q.
					\end{align}
					The second condition in \eqref{eqn-hemi-discrete} gives
				\begin{align}\label{eqn-nume-4}
					d(\bu_h,q_h)=0. 
				\end{align}
				Using the bilinearity of $d(\cdot,\cdot)$, we write 
				\begin{align}\label{eqn-nume-5}
					d(\bu_h,q_h)&=d(\bu_h-\overline{\bu},q_h)+d(\overline{\bu},q_h)\nonumber\\&=d(\bu_h-\overline{\bu},q_h-q)+d(\bu_h-\overline{\bu},q)+d(\overline{\bu},q_h). 
				\end{align}
				Since every weakly convergent sequences are bounded, we know that $\|\bu_h-\bu\|_{\V}$ is bounded and using the fact that $\|q_h-q\|_Q\to 0$ as $h\to 0$, we deduce 
				\begin{align}\label{eqn-nume-6}
					|d(\bu_h-\overline{\bu},q_h-q)|\leq C_k\|\bu_h-\overline{\bu}\|_{\V}\|q_h-q\|_{Q}\to 0\ \text{ as } h\to 0. 
				\end{align}
				Since 	$\bu_h\xrightarrow{w}\overline{\bu}\ \text{ in }\ \V$ (see \eqref{eqn-nume-2}), we immediately have 
				\begin{align}\label{eqn-nume-7}
					d(\bu_h-\overline{\bu},q)\to 0\ \text{ as }\ h\to 0. 
				\end{align}
				Similarly,  the convergence $q_h\to q\ \text{ in } \ Q$ (see \eqref{eqn-nume-3}) implies 
				\begin{align}\label{eqn-nume-8}
					d(\overline{\bu},q_h)\to d(\overline{\bu},q)\ \text{ as }\ h\to 0. 
				\end{align}
				Taking limit $h\to 0$ in \eqref{eqn-nume-5} and then using \eqref{eqn-nume-6}-\eqref{eqn-nume-8}, one can deduce \eqref{eqn-nume-2}. 
				
					\vskip 0.2cm
				\noindent\textbf{Step 4:} Let us now prove the strong convergence 
				\begin{align}\label{eqn-nume-9}
					\bu_h\to\overline{\bu}\ \text{ in }\ \V\ \text{ as }\ h\to 0. 
				\end{align}
				Using Hypothesis \ref{hyp-discrete} $(H(\V_h))$ and $(H(Q_h))$, we infer the existence of sequences $\overline{\bu}_h\in\V_h$ and $\overline{p}_h\in Q_h$ such that 
				\begin{align}\label{eqn-nume-10}
					\overline{\bu}_h\to\overline{\bu} \ \text{ in }\ \V \ \text{ and }\  \overline{p}_h\to\overline{p}\ \text{ in }\ Q \text{ as }\ h\to 0. 
				\end{align}
				The equality \eqref{eqn-a-est-2} and  the estimate \eqref{2.23} lead to 
				\begin{align}\label{eqn-nume-11}
					&2\mu\|\overline{\bu}-\bu_h\|_{\V}^2+\alpha\|\overline{\bu}-\bu_h\|_{\H}^2+\frac{\beta}{2}\||\overline{\bu}|^{\frac{r-1}{2}}(\overline{\bu}-\bu_h)\|_{\H}^2+\frac{\beta}{2}\||{\bu}_h|^{\frac{r-1}{2}}(\overline{\bu}-\bu_h)\|_{\H}^2\nonumber\\&\leq \mu a(\overline{\bu}-\bu_h,\overline{\bu}-\bu_h)+\alpha a_0(\overline{\bu}-\bu_h,\overline{\bu}-\bu_h) +\beta[c(\overline{\bu},\overline{\bu}-\bu_h)-c({\bu}_h,\overline{\bu}-\bu_h)]\nonumber\\&= \mu a(\overline{\bu},\overline{\bu}-\bu_h)-\mu a(\bu_h,\overline{\bu}-\overline{\bu}_h)-\mu a(\bu_h,\overline{\bu}_h-{\bu}_h)\nonumber\\&\quad+\alpha a_0(\overline{\bu},\overline{\bu}-\bu_h)-\alpha a_0(\bu_h,\overline{\bu}-\overline{\bu}_h)-\alpha a_0(\bu_h,\overline{\bu}_h-{\bu}_h)\nonumber\\&\quad+\beta c(\overline{\bu},\overline{\bu}-\bu_h)-\beta c({\bu}_h,\overline{\bu}-\overline{\bu}_h)-\beta c({\bu}_h,\overline{\bu}_h-\bu_h).
				\end{align}
				Taking $\bv_h=\overline{\bu}_h-{\bu}_h$ in \eqref{eqn-hemi-discrete}, we find 
				\begin{align}\label{eqn-nume-12}
					-&\mu a(\bu_h,\overline{\bu}_h-{\bu}_h)-\alpha a_0(\bu_h,\overline{\bu}_h-{\bu}_h) - \beta c(\bu_h,\overline{\bu}_h-{\bu}_h)\nonumber\\&\leq b(\bu_h,\bu_h,\overline{\bu}_h-{\bu}_h)+\kappa c_0(\bu_h,\overline{\bu}_h-{\bu}_h)+d(\overline{\bu}_h-{\bu}_h,p_h)\nonumber\\&\quad+\int_{\Gamma_1}j^0(\bu_{h,\tau};\overline{\bu}_{h,\tau}-{\bu}_{h,\tau})\d S -\langle\f_h,\overline{\bu}_h-{\bu}_h\rangle. 
				\end{align}
				Utilizing \eqref{eqn-nume-12} in \eqref{eqn-nume-11}, we get 
					\begin{align}\label{eqn-nume-13}
				&2\mu\|\overline{\bu}-\bu_h\|_{\V}^2+\alpha\|\overline{\bu}-\bu_h\|_{\H}^2+\frac{\beta}{2}\||\overline{\bu}|^{\frac{r-1}{2}}(\overline{\bu}-\bu_h)\|_{\H}^2+\frac{\beta}{2}\||{\bu}_h|^{\frac{r-1}{2}}(\overline{\bu}-\bu_h)\|_{\H}^2\nonumber\\&\leq  \mu a(\overline{\bu},\overline{\bu}-\bu_h)-\mu a(\bu_h,\overline{\bu}-\overline{\bu}_h) +\alpha a_0(\overline{\bu},\overline{\bu}-\bu_h)-\alpha a_0(\bu_h,\overline{\bu}-\overline{\bu}_h)\nonumber\\&\quad +\beta c(\overline{\bu},\overline{\bu}-\bu_h)-\beta c({\bu}_h,\overline{\bu}-\overline{\bu}_h) +b(\bu_h,\bu_h,\overline{\bu}_h-{\bu}_h)\nonumber\\&\quad+\kappa c_0(\bu_h,\overline{\bu}_h-{\bu}_h)+d(\overline{\bu}_h-{\bu}_h,p_h)+\int_{\Gamma_1}j^0(\bu_{h,\tau};\overline{\bu}_{h,\tau}-{\bu}_{h,\tau})\d S\nonumber\\&\quad-\langle\f_h,\overline{\bu}_h-{\bu}_h\rangle. 
					\end{align}
					Next, we consider each term on the right side of \eqref{eqn-nume-13}. Note that  $\bu_h\xrightarrow{w}\overline{\bu}$ in $\V$, $\bu_h\to\overline{\bu}$ in $\H$  (see \eqref{eqn-nume-1} and \eqref{eqn-nume-1-1}) and Sobolev's embedding yield 
					\begin{align*}
						&a(\overline{\bu},\overline{\bu}-\bu_h)\to 0,\ a_0(\overline{\bu},\overline{\bu}-\bu_h)\to 0 \ \text{ and }\ c(\overline{\bu},\overline{\bu}-\bu_h)\to 0 \ \text{ as } \ h\to 0. 
					\end{align*}
					The convergence  $\|\overline{\bu}_h-\overline{\bu}\|_{\V}\to 0$ (see  \eqref{eqn-nume-10}), Korn's inequality and the fact that  $\|{\bu}_h\|_{\H}$ is bounded independent of $h$ yield
					\begin{align*}
						|a_0(\bu_h,\overline{\bu}-\overline{\bu}_h)|&\leq\|\bu_h\|_{\H}\|\overline{\bu}-\overline{\bu}_h\|_{\H}\leq C_k\|\bu_h\|_{\H}\|\overline{\bu}-\overline{\bu}_h\|_{\V}\to 0\ \text{ as }\ h\to 0. 
					\end{align*}
					Let us rewrite $b(\bu_h,\bu_h,\overline{\bu}_h-{\bu}_h)$ using \eqref{eqn-b-est-4} as 
					\begin{align*}
						b(\bu_h,\bu_h,\overline{\bu}_h-{\bu}_h)=b(\bu_h,\bu_h,\overline{\bu}_h-\overline{\bu})+b(\bu_h,\bu_h,\overline{\bu}). 
					\end{align*}
					Using the fact that $b(\overline{\bu},\overline{\bu},\overline{\bu})=0$,  $\|{\bu}_h\|_{\V}$ is bounded independent of $h,$ \eqref{eqn-b-est-1} and the convergences  \eqref{eqn-nume-1} and \eqref{eqn-nume-1-1}, we find 
					\begin{align*}
						|b(\bu_h,\bu_h,\overline{\bu})|&=|b(\bu_h,\bu_h,\overline{\bu})-b(\overline{\bu},\overline{\bu},\overline{\bu})|\nonumber\\&\leq |b(\bu_h-\overline{\bu},\bu_h,\overline{\bu})|+|b(\overline{\bu},\bu_h-\overline{\bu},\overline{\bu})|
\nonumber\\&\leq C_kC_g \|\bu_h-\overline{\bu}\|_{\H}^{1-\frac{d}{4}}\|\bu_h-\overline{\bu}\|_{\H^1}^{\frac{d}{4}}\|\bu_h\|_{\V}\|\overline{\bu}\|_{\L^4}+|b(\overline{\bu},\bu_h-\overline{\bu},\overline{\bu})|\nonumber\\&\to 0\ \text{ as }\ h\to0. 
						\end{align*}
						We rewrite $c_0(\bu_h,\overline{\bu}_h-{\bu}_h)$ as 
						\begin{align*}
							c_0(\bu_h,\overline{\bu}_h-{\bu}_h)=c_0(\bu_h,\overline{\bu}_h-\overline{\bu})+c_0(\bu_h,\overline{\bu}-{\bu}_h). 
						\end{align*}
						Since $\|\overline{\bu}_h-\overline{\bu}\|_{\V}\to 0$ (see \eqref{eqn-nume-10}) and $\|{\bu}_h\|_{\V}$ is bounded independent of $h$, using \eqref{eqn-a-est-1}, \eqref{eqn-b-est-1} and Sobolev's inequality,  we have as $h\to 0$
						\begin{align*}
							|a(\bu_h,\overline{\bu}-\overline{\bu}_h)|&\leq 2\|\bu_h\|_{\V}\|\overline{\bu}-\overline{\bu}_h\|_{\V}\to 0,\nonumber\\
								|a_0(\bu_h,\overline{\bu}-\overline{\bu}_h)|&\leq\|\bu_h\|_{\H}\|\overline{\bu}-\overline{\bu}_h\|_{\H}\leq C_k^2\|\bu_h\|_{\V}\|\overline{\bu}-\overline{\bu}_h\|_{\V}\to 0,\nonumber\\
							|b(\bu_h,\bu_h,\overline{\bu}_h-\overline{\bu})|&\leq C_b\|\bu_h\|_{\V}^2\|\overline{\bu}_h-\overline{\bu}\|_{\V}\to 0, \nonumber\\
							|c({\bu}_h,\overline{\bu}-\overline{\bu}_h)|&\leq \|\bu_h\|_{\L^{r+1}}^{r}\|\overline{\bu}-\overline{\bu}_h\|_{\L^{r+1}}\leq C_s^{r+1}\|\bu_h\|_{\V}^r\|\overline{\bu}-\overline{\bu}_h\|_{\V}\to 0, \nonumber\\
							|c_0({\bu}_h,\overline{\bu}-\overline{\bu}_h)|&\leq\|\bu_h\|_{\L^{q+1}}^{q}\|\overline{\bu}-\overline{\bu}_h\|_{\L^{q+1}}\leq |\mathcal{O}|^{\frac{r-q}{r+1}}C_s^{q+1}\|\bu_h\|_{\V}^q\|\overline{\bu}-\overline{\bu}_h\|_{\V}\to 0.
						\end{align*}
						We use the convergence \eqref{eqn-nume-1-1}, boundedness of $\{\bu_h\}$ in $\V$,  Sobolev's  and interpolation inequalities to estimate $|c_0(\bu_h,\overline{\bu}-{\bu}_h)|$ as 
						\begin{align*}
						&	|c_0(\bu_h,\overline{\bu}-{\bu}_h)|\nonumber\\&\leq\|\bu_h\|_{\L^{q+1}}^{q}\|\overline{\bu}-{\bu}_h\|_{\L^{q+1}}\nonumber\\&\leq |\mathcal{O}|^{\frac{q(r-q)}{(q+1)(r+1)}}\|\bu_h\|_{\L^{r+1}}^q\|\overline{\bu}-{\bu}_h\|_{\L^{r+1}}^{\frac{(r+1)(q-1)}{(r-1)(q+1)}}\|\overline{\bu}-{\bu}_h\|_{\H}^{\frac{2(r-q)}{(r-1)(q+1)}}\nonumber\\&\leq C_s^{q+\frac{(r+1)(q-1)}{(r-1)(q+1)}} |\mathcal{O}|^{\frac{q(r-q)}{(q+1)(r+1)}}\|\bu_h\|_{\V}^q\left(\|\overline{\bu}\|_{\V}^{\frac{(r+1)(q-1)}{(r-1)(q+1)}}+\|{\bu}_h\|_{\V}^{\frac{(r+1)(q-1)}{(r-1)(q+1)}}\right)\|\overline{\bu}-{\bu}_h\|_{\H}^{\frac{2(r-q)}{(r-1)(q+1)}}\nonumber\\&\to 0\ \text{ as }\ h\to 0. 
						\end{align*}
						We use \eqref{eqn-hemi-discrete}$_2$ and \eqref{eqn-nume-2} to rewrite $d(\overline{\bu}_h-{\bu}_h,p_h)$ as 
						\begin{align*}
							d(\overline{\bu}_h-{\bu}_h,p_h)&=d(\overline{\bu}_h,p_h)-d({\bu}_h,p_h)=d(\overline{\bu}_h,p_h)=d(\overline{\bu}_h-\overline{\bu},p_h)\nonumber\\&\leq C_k\|\overline{\bu}_h-\overline{\bu}\|_{\V}\|p_h\|_{Q}\to 0\ \text{ as }\ h\to 0,
						\end{align*}
						since $\|\overline{\bu}_h-\overline{\bu}\|_{\V}\to 0$ as $h\to 0$ (see \eqref{eqn-nume-10}) and $\{\|p_h\|_Q\}$ is bounded independent of $h$. 	Furthermore, the compact embedding $\mathcal{V} \hookrightarrow \L^2(\Gamma_1)$ (see \cite[Theorem 6.2, Chapter 2]{JNe}) provides  	
						\begin{align*}
							\mbox{$\bu_{h,\tau} \to \overline{\bu}_{\tau} \ \text{ in }\ \L^2(\Gamma_1)$ \  and \ $\overline{\bu}_{h,\tau} \to \overline{\bu}_{\tau} \ \text{ in }\ \L^2(\Gamma_1)$}
						\end{align*}
						 along subsequences still denoted by the same symbol. The above convergences immediately imply 
						\begin{align*}
							\|\overline{\bu}_{h,\tau}-\bu_{h,\tau}\|_{\L^2(\Gamma_1)}\leq 	\|\overline{\bu}_{h,\tau}-\overline{\bu}_{\tau}\|_{\L^2(\Gamma_1)}+\|\overline{\bu}_{\tau}-\bu_{h,\tau}\|_{\L^2(\Gamma_1)}\to 0\ \text{ as } \ h\to 0. 
						\end{align*}
						Therefore, using \eqref{eqn-trace} and \eqref{eqn-j0-est}, we deduce 
						\begin{align*}
							\int_{\Gamma_1}j^0(\bu_{h,\tau};\overline{\bu}_{h,\tau}-{\bu}_{h,\tau})\d S&\leq \int_{\Gamma_1}\left(k_0+k_1|\bu_{h,\tau}|\right)|\overline{\bu}_{h,\tau}-{\bu}_{h,\tau}|\d S\nonumber\\&\leq \left(k_0|\Gamma_1|^{1/2}+k_1\|\bu_{h,\tau}\|_{\L^2(\Gamma_1)}\right)\|\overline{\bu}_{h,\tau}-{\bu}_{h,\tau}\|_{\L^2(\Gamma_1)}\nonumber\\&\leq \left(k_0|\Gamma_1|^{1/2}+\lambda_0^{-1/2}k_1\|\bu_h\|_{\V}\right)\|\overline{\bu}_{h,\tau}-{\bu}_{h,\tau}\|_{\L^2(\Gamma_1)}\nonumber\\&\to 0\ \text{ as } \ h\to 0. 
						\end{align*}
						Finally, we consider $\langle\f_h,\overline{\bu}_h-{\bu}_h\rangle$ and estimate it using the fact that $\f_h\to\f$ in $\V^*,$ and the convergences \eqref{eqn-nume-1} and \eqref{eqn-nume-10} as 
						\begin{align*}
							|\langle\f_h,\overline{\bu}_h-{\bu}_h\rangle|&\leq |\langle\f_h,\overline{\bu}_h-\overline{\bu}\rangle+|\langle\f_h-\f,\overline{\bu}-{\bu}_h\rangle|+|\langle\f,\overline{u}-\bu_h\rangle|\nonumber\\&\leq\|\f_h\|_{\V^*}\|\overline{\bu}_h-\overline{\bu}\|_{\V}+\left(\|\overline{\bu}\|_{\V}+\|\bu_h\|_{\V}\right)\|\f_h-\f\|_{\V^*}+|\langle\f,\overline{\bu}-\bu_h\rangle|\nonumber\\&\to 0\ \text{ as } h\to 0. 
						\end{align*}
						Therefore, the strong convergence of $	\|\bu_h-\bu\|_{\V}\to 0$ as $h\to 0$  follows from \eqref{eqn-nume-13}. 
							\vskip 0.2cm
						\noindent\textbf{Step 5:} Our next aim is to show that $(\overline{\bu},\overline{p})$ is the unique solution to Problem \ref{prob-hemi-1}. For any given $\bv\in\V$ and $q\in Q$, by Hypothesis \ref{hyp-discrete} $(H(\V_h))$ and $(H(Q_h))$, we infer the existence of sequences $\bv_h\in\V$ and $q_h\in Q$  such that 
						\begin{align}\label{eqn-numer-11}
							\bv_h\to\bv\ \text{ in }\ \V\ \text{ and }\ q_h\to q\ \text{ in }\ Q. 
						\end{align} 
						One can deduce  from \eqref{eqn-hemi-discrete}$_1$ that 
						\begin{align}\label{eqn-nume-14}
								\mu a(\bu_h,\bv_h)+b(\bu_h,\bu_h,\bv_h)+\alpha a_0(\bu_h,\bv_h)&+\beta c(\bu_h,\bv_h)+\kappa c_0(\bu_h,\bv_h)+d(\bv_h,p_h)\nonumber\\+\int_{\Gamma_1}j^0(\bu_{h,\tau};\bv_{h,\tau})\d S&\geq 	\langle\f_h,\bv_h\rangle.
						\end{align}
						Using the convergences \eqref{eqn-nume-9} and \eqref{eqn-numer-11}, we deduce as $h\to 0$
						\begin{align*}
							|a(\bu_h,\bv_h)-a(\overline{\bu},\bv)|&\leq| a(\bu_h-\overline{\bu},\bv_h)|+|a(\overline{\bu},\bv_h-\bv)|\nonumber\\&\leq 2\mu\|\bu_h-\overline{\bu}\|_{\V}\|\bv_h\|_{\V}+2\mu\|\overline{\bu}\|_{\V}\|\bv_h-\bv\|_{\V}\to 0,\nonumber\\	|a_0(\bu_h,\bv_h)-a_0(\overline{\bu},\bv)| &\leq| a_0(\bu_h-\overline{\bu},\bv_h)|+|a_0(\overline{\bu},\bv_h-\bv)|\nonumber\\&\leq C_k^2\|\bu_h-\overline{\bu}\|_{\V}\|\bv_h\|_{\V}+C_k^2\|\overline{\bu}\|_{\V}\|\bv_h-\bv\|_{\V}\to 0,\nonumber\\ |b(\bu_h,\bu_h,\bv_h)-b(\overline{\bu},\overline{\bu},\bv)|&\leq |b(\bu_h-\overline{\bu},\bu_h,\bv_h)|+|b(\overline{\bu},\bu_h-\overline{\bu},\bv_h)|\nonumber\\&\quad+|b(\overline{\bu},\overline{\bu},\bv_h-\bv)|\nonumber\\&\leq C_b\|\bu_h-\overline{\bu}\|_{\V}\|\bu_h\|_{\V}\|\bv_h\|_{\V}+C_b\|\overline{\bu}\|_{\V}\|\bu_h-\overline{\bu}\|_{\V}\|\bv_h\|_{\V}\nonumber\\&\quad+C_b\|\overline{\bu}\|_{\V}\|\overline{\bu}\|_{\V}\|\bv_h-\bv\|_{\V}\to 0,\nonumber\\
							|c(\bu_h,\bv_h)-c(\overline{\bu},\bv)|&\leq |c(\bu_h,\bv_h)-c(\overline{\bu},\bv_h)|+|c(\overline{\bu},\bv_h-\bv)|\nonumber\\&\leq\left\langle\int_0^1\mathcal{C}'(\theta\bu_h+(1-\theta)\overline{\bu})(\bu_h-\overline{\bu})\d\theta,\bv_h\right\rangle\nonumber\\&\quad+\|\overline{\bu}\|_{\L^{r+1}}^{r}\|\bv_h-\bv\|_{\L^{r+1}}\nonumber\\&\leq r\left(\|\bu_h\|_{\L^{r+1}}+\|\overline{\bu}\|_{\L^{r+1}}\right)^{r-1}\|\bu_h-\overline{\bu}\|_{\L^{r+1}}\|\bv_h\|_{\L^{r+1}}\nonumber\\&\quad+\|\overline{\bu}\|_{\L^{r+1}}^{r}\|\bv_h-\bv\|_{\L^{r+1}}\nonumber\\&\leq rC_s^{r+1}\left(\|\bu_h\|_{\V}+\|\overline{\bu}\|_{\V}\right)^{r-1}\|\bu_h-\overline{\bu}\|_{\V}\|\bv_h\|_{\V}\nonumber\\&\quad+C_s^{r+1}\|\overline{\bu}\|_{\V}^{r}\|\bv_h-\bv\|_{\V}\to 0,\nonumber\\ 	|c_0(\bu_h,\bv_h)-c_0(\overline{\bu},\bv)|&\leq |c_0(\bu_h,\bv_h)-c_0(\overline{\bu},\bv_h)|+|c_0(\overline{\bu},\bv_h-\bv)|\nonumber\\&\leq\left\langle\int_0^1\mathcal{C}_0'(\theta\bu_h+(1-\theta)\overline{\bu})(\bu_h-\overline{\bu})\d\theta,\bv_h\right\rangle\nonumber\\&\quad+\|\overline{\bu}\|_{\L^{q+1}}^{q}\|\bv_h-\bv\|_{\L^{q+1}}\nonumber\\&\leq q\left(\|\bu_h\|_{\L^{q+1}}+\|\overline{\bu}\|_{\L^{q+1}}\right)^{q-1}\|\bu_h-\overline{\bu}\|_{\L^{q+1}}\|\bv_h\|_{\L^{q+1}}\nonumber\\&\quad+\|\overline{\bu}\|_{\L^{q+1}}^{q}\|\bv_h-\bv\|_{\L^{q+1}}\nonumber\\&\leq qC_s^{q+1}|\mathcal{O}|^{\frac{r-q}{r+1}}\left(\|\bu_h\|_{\V}+\|\overline{\bu}\|_{\V}\right)^{q-1}\|\bu_h-\overline{\bu}\|_{\V}\|\bv_h\|_{\V}\nonumber\\&\quad+C_s^{q+1}\|\overline{\bu}\|_{\V}^{q}\|\bv_h-\bv\|_{\V}\to 0,\nonumber\\
							|	\langle\f_h,\bv_h\rangle-	\langle\f,\bv\rangle|&\leq|\langle\f_h-\f,\bv_h\rangle|+|	\langle\f,\bv_h-\bv\rangle|\nonumber\\&\leq\|\f_h-\f\|_{\V^*}\|\bv_h\|_{\V}+\|\f\|_{\V^*}\|\bv_h-\bv\|_{\V}\to 0. 
						\end{align*}
						Once again the the compact embedding $\mathcal{V} \hookrightarrow \L^2(\Gamma_1)$ provides (along a subsequence denoted by the same symbol) 
						\begin{align*}
							\bu_{h,\tau}\to\overline{\bu}_{\tau}\ \text{ and }\ \bv_{h,\tau}\to{\bv}_{\tau}\ \text{ a.e. on }\ \Gamma_1. 
						\end{align*}
						Therefore, we infer from Proposition \ref{prop-sub-diff} (ii) that 
						\begin{align*}
							\limsup_{h\to 0}\int_{\Gamma_1}j^0(\bu_{h,\tau};\bv_{h,\tau})\d S\leq\int_{\Gamma_1} \limsup_{h\to 0}j^0(\bu_{h,\tau};\bv_{h,\tau})\d S\leq \int_{\Gamma_1}j^0(\overline{\bu}_{\tau};\bv_{\tau})\d S.
						\end{align*}
						Taking limit supremum over $h\to 0$ in \eqref{eqn-nume-14}, we arrive at 
							\begin{align}\label{eqn-nume-15}
							\mu a(\overline{\bu},\bv)+b(\overline{\bu},\overline{\bu},\bv)+\alpha a_0(\overline{\bu},\bv)&+\beta c(\overline{\bu},\bv)+\kappa c_0(\overline{\bu},\bv)+d(\bv,\overline{p})\nonumber\\+\int_{\Gamma_1}j^0(\overline{\bu}_{\tau};\bv_{\tau})\d S&\geq 	\langle\f,\bv\rangle.
						\end{align}
						The relations \eqref{eqn-nume-15} and \eqref{eqn-nume-2} show that $(\overline{\bu},\overline{p})=(\bu,p)$ is the unique solution of Problem \ref{prob-hemi-1}. Since the limit $(\bu,p)$ is unique, the entire sequence $\{(\bu_h,p_h)\}_h$ converges and  as $h\to 0$ 
						\begin{align*}
							\bu_h\to\bu\ \text{ in } \ \V\ \text{ and }\ q_h\to q\ \text{ in }\ Q. 
						\end{align*}
							\vskip 0.2cm
						\noindent\textbf{Step 6:} Finally, we show the following strong convergence: 
						\begin{align}\label{eqn-numer-1-6}
							p_h\to p\ \text{ in }\ Q. 
						\end{align}
						We infer from the Babu\u{s}ka-Brezzi condition \eqref{eqn-numer-1} that 
						\begin{align}\label{eqn-numer-1-7}
							\vartheta_1\|p_h-\overline{p}_h\|_Q\leq\sup_{\bv_h\in\V_{0,h}}\frac{b(\bv_h,p_h-\overline{p}_h)}{\|\bv_h\|_{\V}}. 
						\end{align}
						Let us rewrite $d(\bv_h,p_h-\overline{p}_h)$ as 
						\begin{align}\label{eqn-numer-1-8}
							d(\bv_h,p_h-\overline{p}_h)=d(\bv_h,p_h-p)+b(\bv_h,p-\overline{p}_h). 
						\end{align}
						We infer from \eqref{eqn-numer-0} and \eqref{eqn-con-1} that 
							\begin{align}\label{eqn-numer-16}
							\mu a(\bu_h,\bv_h)+b(\bu_h,\bu_h,\bv_h)&+\alpha a_0(\bu_h,\bv_h)+\beta c(\bu_h,\bv_h)+\kappa c_0(\bu_h,\bv_h)\nonumber\\+d(\bv_h,p_h)&=\langle\f_h,\bv_h\rangle \ \text{ for all }\ \bv_h\in\V_{0,h},
						\end{align}
						and 
							\begin{align}\label{eqn-numer-17}
							\mu a(\bu,\bv_h)+b(\bu,\bu,\bv_h)&+\alpha a_0(\bu,\bv_h)+\beta c(\bu,\bv_h)+\kappa c_0(\bu,\bv_h)\nonumber\\+d(\bv_h,p)&=	\langle\f,\bv_h\rangle  \ \text{ for all }\ \bv_h\in\V_{0,h}. 
						\end{align}
						Subtracting \eqref{eqn-numer-17} from \eqref{eqn-numer-16}, we deduce 
						\begin{align}\label{eqn-numer-18}
							d(\bv_h,p_h-p)&=\mu a(\bu-\bu_h,\bv_h)+[b(\bu,\bu,\bv_h)-b(\bu_h,\bu_h,\bv_h)]+\alpha a_0(\bu-\bu_h,\bv_h)\nonumber\\&\quad+\beta [c(\bu,\bv_h)-c(\bu_h,\bv_h)]+\kappa[c_0(\bu,\bv_h)-c_0(\bu_h,\bv_h)]+\langle\f_h-\f,\bv_h\rangle.
						\end{align}
						Using \eqref{eqn-equiv-1}, \eqref{eqn-a-est-1}, \eqref{eqn-b-est-1}, Taylor's formula Sobolev's embedding and H\"older's inequality, we find 
						\begin{align*}
							|a(\bu-\bu_h,\bv_h)|&\leq 2\|\bu-\bu_h\|_{\V}\|\bv_h\|_{\V},\nonumber\\
							|a_0(\bu-\bu_h,\bv_h)|&\leq \|\bu-\bu_h\|_{\H}\|\bv_h\|_{\H}\leq C_k^2\|\bu-\bu_h\|_{\V}\|\bv_h\|_{\V},\nonumber\\
						|b(\bu,\bu,\bv_h)-b(\bu_h,\bu_h,\bv_h)|&\leq |b(\bu-\bu_h,\bu,\bv_h)|+|b(\bu_h,\bu-\bu_h,\bv_h)|\nonumber\\ &\leq C_b\left(\|\bu\|_{\V}+\|\bu_h\|_{\V}\right)\|\bu-\bu_h\|_{\V}\|\bv_h\|_{\V},\nonumber\\
						|c(\bu,\bv_h)-c(\bu_h,\bv_h)|&=\left|\left<\int_0^1\mathcal{C}'(\theta\bu+(1-\theta)\bu_h)\d\theta(\bu-\bu_h),\bv_h\right>\right|\nonumber\\&\leq r\left(\|\bu\|_{\L^{r+1}}+\|\bu_h\|_{\L^{r+1}}\right)^{r-1}\|\bu-\bu_h\|_{\L^{r+1}}\|\bv_h\|_{\L^{r+1}}\nonumber\\&\leq rC_s^{r+1}\left(\|\bu\|_{\V}+\|\bu_h\|_{\V}\right)^{r-1}\|\bu-\bu_h\|_{\V}\|\bv_h\|_{\V},\\
							|c_0(\bu,\bv_h)-c_0(\bu_h,\bv_h)|&=\left|\left<\int_0^1\mathcal{C}'_0(\theta\bu+(1-\theta)\bu_h)\d\theta(\bu-\bu_h),\bv_h\right>\right|\nonumber\\&\leq q\left(\|\bu\|_{\L^{q+1}}+\|\bu_h\|_{\L^{q+1}}\right)^{q-1}\|\bu-\bu_h\|_{\L^{q+1}}\|\bv_h\|_{\L^{q+1}}\nonumber\\&\leq qC_s^{q+1}|\mathcal{O}|^{\frac{r-q}{r+1}}\left(\|\bu\|_{\V}+\|\bu_h\|_{\V}\right)^{q-1}\|\bu-\bu_h\|_{\V}\|\bv_h\|_{\V},\nonumber\\
							|\langle\f_h-\f,\bv_h\rangle|&\leq\|\f_h-\f\|_{\V^*}\|\bv_h\|_{\V}. 
						\end{align*} 
					By combining the above estimates with \eqref{eqn-numer-18} and \eqref{eqn-numer-1-7}, and then using \eqref{eqn-numer-1-8}, we conclude that
					\begin{align}\label{eqn-numer-19}
						&	\vartheta_1\|p_h-\overline{p}_h\|_Q\nonumber\\&\leq\sup_{\bv_h\in\V_{0,h}}\frac{d(\bv_h,p_h-p)+d(\bv_h,p-\overline{p}_h)}{\|\bv_h\|_{\V}}\nonumber\\&\leq \sup_{\bv_h\in\V_{0,h}}\frac{\begin{array}{l}\mu a(\bu-\bu_h,\bv_h)+[b(\bu,\bu,\bv_h)-b(\bu_h,\bu_h,\bv_h)]+\alpha a_0(\bu-\bu_h,\bv_h)\\ \quad+\beta [c(\bu,\bv_h)-c(\bu_h,\bv_h)]+\kappa[c_0(\bu,\bv_h)-c_0(\bu_h,\bv_h)]+\langle\f_h-\f,\bv_h\rangle\\ \quad+d(\bv_h,p-\overline{p}_h)\end{array}}{\|\bv_h\|_{\V}}\nonumber\\&\leq \Big[\Big\{(2\mu+\alpha C_k^2)+C_b(\|\bu\|_{\V}+\|\bu_h\|_{\V})+rC_s^{r+1}\left(\|\bu\|_{\V}+\|\bu_h\|_{\V}\right)^{r-1}\nonumber\\&\quad+qC_s^{q+1}|\mathcal{O}|^{\frac{r-q}{r+1}}\left(\|\bu\|_{\V}+\|\bu_h\|_{\V}\right)^{q-1}\Big\}\|\bu-\bu_h\|_{\V}+\|\f_h-\f\|_{\V^*}+C_k\|p-\overline{p}_h\|_{Q}\Big].
					\end{align}
					An application of triangle inequality yields 
					\begin{align}\label{eqn-numer-20}
						\|p_h-p\|_Q\leq 	\|p_h-\overline{p}_h\|_Q+	\|\overline{p}_h-p\|_Q. 
					\end{align}
					Using \eqref{eqn-numer-19} in \eqref{eqn-numer-20}, we deduce 
					\begin{align}\label{eqn-numer-21}
						\|p_h-p\|_Q\leq C\left[\|\bu-\bu_h\|_{\V}+\|\f_h-\f\|_{\V^*}+\|p-\overline{p}_h\|_{Q}\right]. 
					\end{align}
				Applying the inequalities \eqref{eqn-nume-9} and \eqref{eqn-nume-10}, and noting that $\overline{\bu} = \bu$, $\overline{p} = p$, along with the convergence $\f_h \to \f$ in $\V^*$, we deduce from \eqref{eqn-numer-21} that \eqref{eqn-numer-1-6} holds.
				\end{proof}
				
				To proceed with the convergence analysis of the numerical approximation for the optimal control problem \eqref{eqn-opt-2}, we introduce an additional assumption on the cost functional $R$. 
				
				\begin{hypothesis}\label{hyp-conv}
					\begin{enumerate}
						\item [$(H(R'))$] If $\bu_m\to\bu$, $p_m\to p$ and $\f_m\to\f$ in $\V^*$, then 
						\begin{align}
							R(\bu,p,\f)=\lim_{m\to\infty}R(\bu_m,p_m,\f_m). 
						\end{align}
					\end{enumerate}
				\end{hypothesis}
			We note that the cost functionals defined in \eqref{eqn-cost-1} and \eqref{eqn-cost-2} satisfy this property. The following result is inspired by \cite[Theorem 4.2]{WWXCWH}, and for the sake of completeness, we include a proof here.
			
			\begin{theorem}\label{thm-conv-cost}
					Let 	Hypotheses \ref{hyp-sup-j}, \ref{hyp-opt} $(H(R))$, \eqref{eqn-unique-con-1} or \eqref{eqn-unique-con-2},  Hypotheses \ref{hyp-discrete}, and in addition Hypothesis \ref{hyp-conv}  be satisfied. For each $h>0$, let $\f_h$ be a solution of the problem \eqref{eqn-opt-2}.
					Then, there exists a subsequence, still denoted by $\{\f_h\}$, and an element $\f \in \V^*_{\mathrm{ad}}$ such that
						$$
					\f_h \to \f \ \text{in} \ \V^*, \quad \bu_h(\f_h) \to \bu(\f) \ \text{in} \ \V, \quad p_h(\f_h) \to p(\f) \ \text{in} \ Q,
					$$
						and $\f \in \V^*_{\mathrm{ad}}$ is a solution to the optimal control problem \eqref{eqn-opt-1}.
			\end{theorem}
			\begin{proof}
				By the assumption $H(\V^*_{\mathrm{ad}},h)$ in Hypotheses \ref{hyp-discrete}, there exists a sequence $\{\f_h'\}_h \subset \V^*_{\mathrm{ad}}$ such that
				$$
				\|\f_h' - \f_h\|_{\V^*} \to 0 \  \text{ as } \ h \to 0.
				$$
				Since $\V^*_{\mathrm{ad}}$ is compact, we can extract a subsequence (if necessary) such that $\f_h' \to \f$ in $\V^*$, for some $\f \in \V^*_{\mathrm{ad}}$. Then, by the triangle inequality, it follows that $\f_h \to \f$ in $\V^*$. Applying Theorem \ref{thm-discrete-conv}, we obtain
				\begin{align*}
					\bu_h\to\bu\ \text{ in }\ \V\ \text{ and }\ p_h\to p\ \text{ in }\ Q. 
				\end{align*}
				Therefore, by Hypothesis \ref{hyp-conv} $(H(R'))$, we have 
				\begin{align}
					\mathcal{R}(\f)=\lim_{h\to 0}\mathcal{R}_h(\f_h). 
				\end{align}
				Our aim is to show that $\f$ is a solution of the optimal control problem \eqref{eqn-opt-1}. Let $\overline{\f}$ be a solution to \eqref{eqn-opt-1}. Then by Hypothesis \ref{hyp-discrete} $(H(\V_{\mathrm{ad},h}^*))$, there exists $\overline{\f}_h\in\V_{\mathrm{ad},h}^*$ such that 
				\begin{align*}
					\overline{\f}_h\to \overline{\f} \ \text{ in }\V^*\ \text{ as }\ h\to 0. 
				\end{align*}
				Therefore, an application of Theorem \ref{thm-discrete-conv} yields 
				\begin{align*}
						\bu_h(\overline{\f}_h)\to\bu(\overline{\f})\ \text{ in }\ \V\ \text{ and }\ p_h(\overline{\f}_h)\to p(\overline{\f})\ \text{ in }\ Q. 
				\end{align*}
				Once again by Hypothesis \ref{hyp-conv} $(H(R'))$,  we deduce 
				\begin{align}
					\mathcal{R}(\overline{\f})=\lim_{h\to 0}\mathcal{R}_h(\overline{\f}_h). 
				\end{align}
				Since $\f_h$ is a solution of the problem \eqref{eqn-opt-2}, we infer 
				\begin{align*}
					\mathcal{R}_h(\f)\leq\mathcal{R}_h(\overline{\f}_h). 
				\end{align*}
				By taking the limit on both sides of the above inequality as $h\to 0$, we arrive at
				\begin{align*}
					\mathcal{R}(\f)\leq\mathcal{R}(\overline{\f}),
				\end{align*}
				so that $\f$ is a solution of the optimal control problem \eqref{eqn-opt-1}.
			\end{proof}

\section{Numerical experiments}\setcounter{equation}{0} \label{sec7}

The theoretical results established in Section~\ref{sec5}-\ref{sec6} guarantee existence (see Theorem \ref{thm-optimal-exis}) of an optimal
control for the hemivariational inequality (HVI) governed CBFeD model and convergence of
discrete approximations (see Theorem~\ref{thm-conv-cost}). However, due to the presence of the nonsmooth slip boundary condition,
the reduced cost functional $\mathcal{R}(\f) = R(\bu(\f),p(\f),\f)$ is in general \emph{not
	differentiable}. Consequently, classical first-order optimality conditions based on adjoint
equations are not available. Instead, numerical solution of the optimal control problem must be
based on \emph{nonsmooth optimization techniques}, typically involving subgradient methods,
bundle methods, or proximal-type iterations (see Clarke~\cite{FHC}).
Here we employ a projected subgradient method, which is simple to implement and sufficient to
illustrate the feasibility of optimal control computations for HVI-CBFeD problems. Let us briefly recall the algorithm for the computation of Problem~\ref{prob-hemi-discrete} by considering $\f_h$ as forcing term, as discussed in \cite[Section 5]{WAMTM}. 
The computational domain is taken as the unit square $\mathcal{O} = (0,1) \times (0,1)$, where a slip boundary condition is imposed on the top boundary 
$\Gamma_{1} = (0,1) \times \{1\}$, while homogeneous Dirichlet boundary conditions are prescribed on the remaining part of the boundary. 
The nonlinear slip condition is derived from the functional
\[
j(\boldsymbol{z}) = \int_{0}^{|\boldsymbol{z}|} \omega(t)\,dt,
\]
where $\omega : [0,\infty) \to \mathbb{R}$ is continuous with $\omega(0) > 0$. 
In our computations, we employ the slip coefficient of the form
\begin{align*}
\omega(t) = (a-b)e^{-\rho t} + b,
\end{align*}
with constants $a > b > 0$ and $\rho > 0$, whose values are specified in Table~\ref{tab:data}. 
The nonlinear slip boundary condition 
\begin{align*}
- \boldsymbol{\sigma}_{\tau} \in \partial j(\bu_{\tau})
\end{align*}
is equivalently expressed as
\begin{align*}
\begin{cases}
	|\boldsymbol{\sigma}_{\tau}| \leq \omega(0), & \text{if } \bu_{\tau} = 0, \\[1ex]
	- \boldsymbol{\sigma}_{\tau} = \omega(|\bu_{\tau}|)\,\dfrac{\bu_{\tau}}{|\bu_{\tau}|}, & \text{if } \bu_{\tau} \neq 0.
\end{cases}
\end{align*}

\begin{algorithm}[Projected Subgradient Scheme]\label{algorithm}
Below we give a concise, implementation-oriented algorithm for the optimal control problem governed by a hemivariational inequality (HVI) for CBFeD. 

%
	\begin{enumerate}
		\item \textbf{Initialization.}
		\begin{enumerate}
			\item Generate mesh and build FE spaces $\V_h,Q_h,$ and $\V_{\mathrm{ad},h}^*$.
			\item Choose initial control $\f_h^0\in \V_{\mathrm{ad},h}^*$ (e.g. $\f_h^0\equiv 0$).
			\item Fix tolerances $\varepsilon_{\rm hvi}$ (state solver), $\varepsilon_{\rm opt}$ (outer stop),
			and a step-size rule $\{\tau_k\}$.
		\end{enumerate}
		\item \textbf{Outer iteration:} for $k=0,1,2,\dots$ until stopping do:
		\begin{enumerate}
			\item[(a)] \textbf{State solve (discrete HVI).} With control $\f_h^k$, solve the discrete HVI for
			$(\bu_h^k,p_h^k)$ using the Uzawa-Newton inner solver as discussed in \cite[Section 5]{WAMTM}.
			\item[(b)] \textbf{Evaluate cost:} Compute
			\begin{align*}
				\mathcal{R}_1(\f_h) & =\frac{\alpha_1}{2}\int_{\mathcal{O}}\|\bu_h^k(\x)-\bu_d(\x)\|^2_{\R^d}\d\x+\frac{\alpha_2}{2}\int_{\mathcal{O}}|p_h^k(\x)-p_d(\x)|^2\d\x \\
				& \qquad +\frac{\alpha_3}{2}\int_{\mathcal{O}}\|\f_h^k(\x)\|_{\R^d}^2\d\x
			\end{align*}
			\item[(c)] \textbf{Compute a generalized subgradient $\boldsymbol{g}_h^k\in\partial \mathcal{R}_1(\f_h^k)$.} Approximate directional derivatives numerically:
				\[
				[\boldsymbol{g}_h^k]_i \approx \frac{\mathcal{R}_1(\f_h^k+\delta e_i) - \mathcal{R}_1(\f_h^k)}{\delta},
				\]
				for a small $\delta>0$.
			\item[(d)] \textbf{Control update (projected subgradient step).}
			\[
			\f_h^{k+1} \;=\; P_{\V_{\mathrm{ad},h}^*}\big(\f_h^k - \tau_k\, \boldsymbol{g}^k\big).
			\]
			Use a step-size rule $\tau_k$ appropriate.
			\item[(e)] \textbf{Stopping check.} If $\|\f_h^{k+1}-\f_h^k\|_{L^2(\Omega;\mathbb{R}^d)}<\varepsilon_{\rm opt},$ then stop; otherwise continue.
		\end{enumerate}
		
	\end{enumerate}
	\end{algorithm}
We implement three numerical examples by considering the functional $\mathcal{R}_1,$ and the numerical experiments are carried out using the finite element library FEniCS, implemented in Python, and the iterative procedure described in the above Algorithm is employed. For the implementation, we choose the following parameters:
\begin{table}[h!]
	\centering
	\begin{tabular}{c|c c c c c c c c c c c c}
		\hline
		\textbf{Example} & $\mu$ & $\alpha$ & $\beta$ & $\kappa$ & $(r,q)$ & $a$ & $b$ & $\rho$ & $\eta$ & $\alpha_1$ & $\alpha_2$ & $\alpha_3$ \\
		\hline
		0.5 & 1.2 & 0   & 0 & 0    & (--,--) & 1.55 & 1.53 & 3.0 & 1   & 1.0 & 1.2 & 0.2 \\
		2 & 1.0 & 1.5 & 1.0   & 0.0 & (3,--)  & 4.01 & 4.00 & 1.5 & 2.0   & 1.0 & 1.0 & 0.5 \\
		3 & 1.0 & 0.5 & 1.0 & -0.5 & (3.0, 1.5)  & 3.25 & 3.20 & 0.5 & 0.1 & 1.0 & 0.5 & 0.1 \\
		\hline
	\end{tabular}
	\caption{Setup for numerical experiments}\label{tab:data}
\end{table}

\subsection{Example 1} It is worth noting that, with the parameter choices specified in 
Table~\ref{tab:data}, this example effectively reduces to the classical \emph{Navier-Stokes equations}.
Here, we choose the desired vector $\bu_d$ and pressure $p_d$ as 
\begin{align*}
	& \bu_d = \begin{pmatrix}
		-\,x^{2}(x-1)\,y\,(3y-2) \\ x\,(3x-2)\,y^{2}(y-1)
	\end{pmatrix}, \  p_d = (2x-1)(2y-1).
\end{align*}
In this example, we initialize the control with the zero vector, 
$\f_h^0 = (x-1,y-1)$. For the Uzawa-Newton inner iterations, the stopping 
tolerance is set to $\varepsilon_{\mathrm{hvi}} = 10^{-5}$. The outer 
iteration, corresponding to the successive control updates, is terminated 
once the criterion $\varepsilon_{\mathrm{opt}} = 10^{-5}$ is satisfied. 
Furthermore, the step-size sequence is chosen as a constant value, $\tau_k = 10^{-2}$ for all $k$. The optimal control $\f_h^*$, together with the corresponding optimal velocity $\bu_h^*$ and pressure $p_h^*$ in the finest mesh, is shown in Figure~\ref{fig:Ex1-plot-oVCP}. The history of the cost functional on the 
finest mesh is presented in Figure~\ref{fig:Ex1-CostHist}. To assess the 
accuracy, we compute the relative errors by taking the numerical solution on the finest mesh as the reference solution. These errors are summarized in Table~\ref{tab:ex1-errors-pred} and illustrated graphically in Figure~\ref{fig:Ex1-errors}. As can be clearly observed, the errors in the velocity, pressure, and control decrease consistently, thereby confirming the expected convergence behavior.

\begin{figure}[ht]
	\centering
	\includegraphics[width=\textwidth]{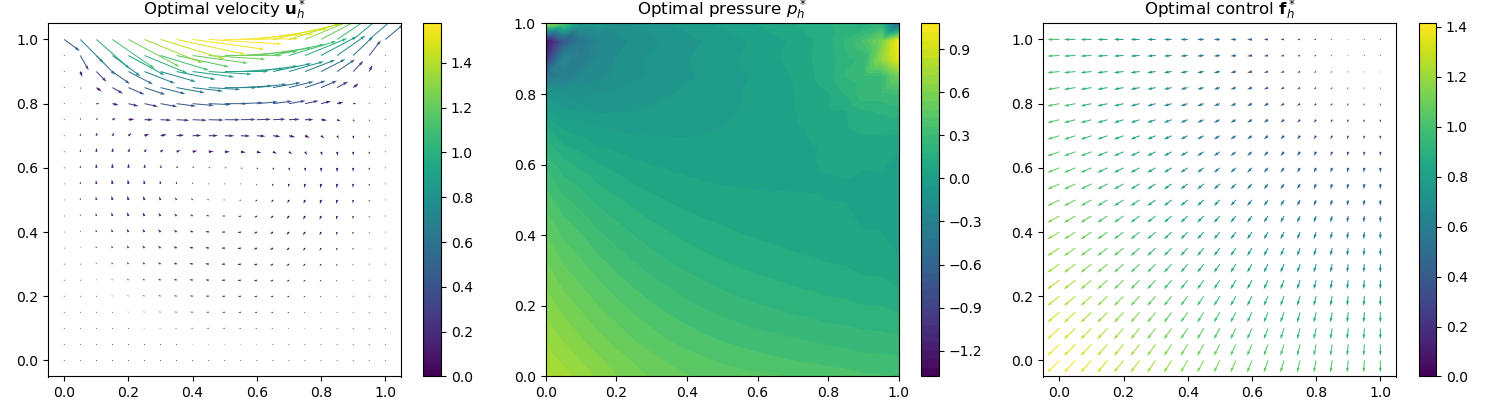}
	\caption{Plot of optimal velocity $\bu_h^*$, optimal pressure $p_h^*$, and optimal control $\f_h^*$ in mesh size $25\times 25$.}
	\label{fig:Ex1-plot-oVCP}
\end{figure}

\begin{figure}[ht]
	\centering
	\begin{subfigure}{0.48\textwidth}
		\centering
		\includegraphics[width=\linewidth]{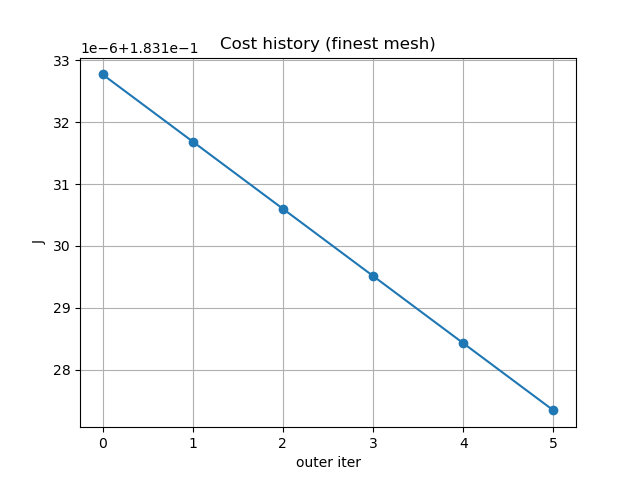}
		\caption{Cost history.}
		\label{fig:Ex1-CostHist}
	\end{subfigure}
	\hfill
	\begin{subfigure}{0.48\textwidth}
		\centering
		\includegraphics[width=\linewidth]{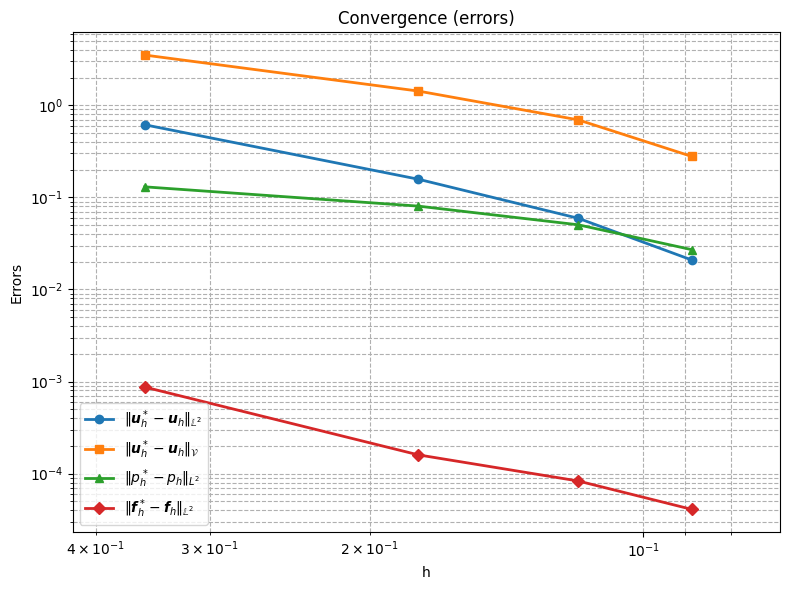}
		\caption{Pressure error vs. mesh size.}
		\label{fig:Ex1-errors}
	\end{subfigure}
	\caption{Cost history and plots of errors.}
	\label{fig:sidebyside}
\end{figure}

\begin{table}[h!]
	\centering
	\begin{tabular}{ccccc}
		\hline
		$h$ & $\|\boldsymbol{u}_h^*-\boldsymbol{u}_h\|_{\mathbb{L}^2}$ &
		$\|\boldsymbol{u}_h^*-\boldsymbol{u}_h\|_{\mathcal{V}}$ &
		$\|p_h^*-p_h\|_{L^2}$ &
		$\|\boldsymbol{f}_h^*-\boldsymbol{f}_h\|_{\mathbb{L}^2}$ \\
		\hline
		$3.5355{\times}10^{-1}$ & $6.144{\times}10^{-1}$ & $3.512{\times}10^{0}$ & $1.301{\times}10^{-1}$ & $8.700{\times}10^{-4}$ \\
		$1.7678{\times}10^{-1}$ & $1.574{\times}10^{-1}$ & $1.427{\times}10^{0}$ & $8.032{\times}10^{-2}$ & $1.597{\times}10^{-4}$ \\
		$1.1785{\times}10^{-1}$ & $5.922{\times}10^{-2}$ & $6.939{\times}10^{-1}$ & $5.033{\times}10^{-2}$ & $8.304{\times}10^{-5}$ \\
		$8.8388{\times}10^{-2}$ & $2.085{\times}10^{-2}$ & $2.780{\times}10^{-1}$ & $2.702{\times}10^{-2}$ & $4.083{\times}10^{-5}$ \\
		\hline
	\end{tabular}
	\caption{Relative errors for velocity in $\mathbb{L}^2$, velocity in $\mathcal{V}$, pressure in $L^2$, and control in $\mathbb{L}^2$.}\label{tab:ex1-errors-pred}
\end{table}

\subsection{Example 2} Here, we choose the desired vector $\bu_d$ and pressure $p_d$ as 
\begin{align*}
	& \bu_d(x,y) = \begin{pmatrix}\sin(\pi x)\sin(\pi y) \\ \cos(\pi x)\cos(\pi y)\end{pmatrix}, \ p_d(x,y) = \sin(\pi x)\cos(\pi y).
\end{align*}
We initialize the control with the zero vector, 
$\f_h^0 = (\sin(\pi x), \cos(\pi y))$. The stopping tolerance for the Uzawa-Newton inner 
iterations is set to $\varepsilon_{\mathrm{hvi}} = 10^{-5}$, while the 
outer stopping criterion for the successive control updates is chosen as 
$\varepsilon_{\mathrm{opt}} = 10^{-5}$. The step-size sequence is again 
taken to be constant with $\tau_k = 10^{-2}$ for all $k$. At the finest mesh, the optimal 
control $\f_h^*$, along with the corresponding optimal velocity $\bu_h^*$ 
and pressure $p_h^*$, is displayed in Figure~\ref{fig:Ex2-plot-oVCP}. 
The cost functional history on the finest mesh is presented in 
Figure~\ref{fig:Ex2-CostHist}. The relative errors are computed by taking 
the numerical solution on the finest mesh as the reference solution. These 
results are reported in Table~\ref{tab:ex2-errors-pred} and depicted in 
Figure~\ref{fig:Ex2-errors}. The errors in velocity, pressure, and control 
are observed to decrease steadily, confirming the expected convergence 
behavior.

\begin{figure}[ht]
	\centering
	\includegraphics[width=\textwidth]{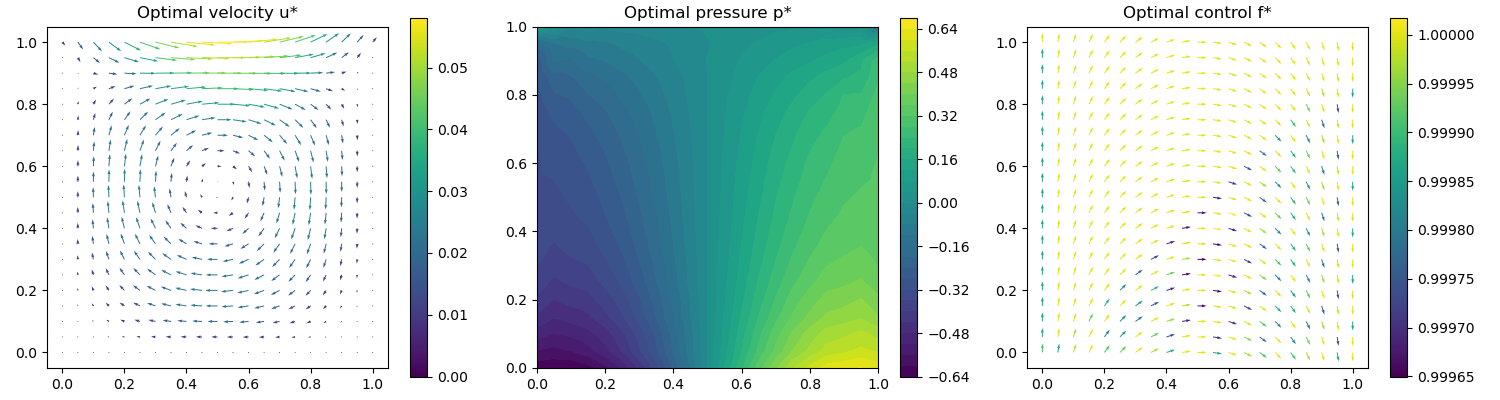}
	\caption{Plot of optimal velocity $\bu_h^*$, optimal pressure $p_h^*$, and optimal control $\f_h^*$ in mesh size $25\times 25$.}
	\label{fig:Ex2-plot-oVCP}
\end{figure}

\begin{figure}[ht]
	\centering
	\begin{subfigure}{0.48\textwidth}
		\centering
		\includegraphics[width=\linewidth]{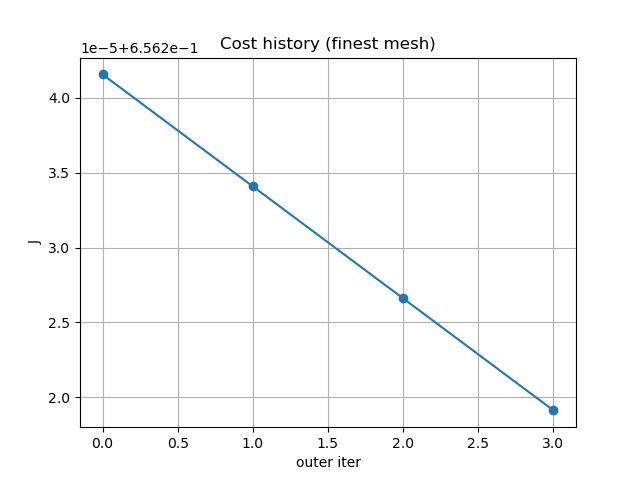}
		\caption{Cost history.}
		\label{fig:Ex2-CostHist}
	\end{subfigure}
	\hfill
	\begin{subfigure}{0.48\textwidth}
		\centering
		\includegraphics[width=\linewidth]{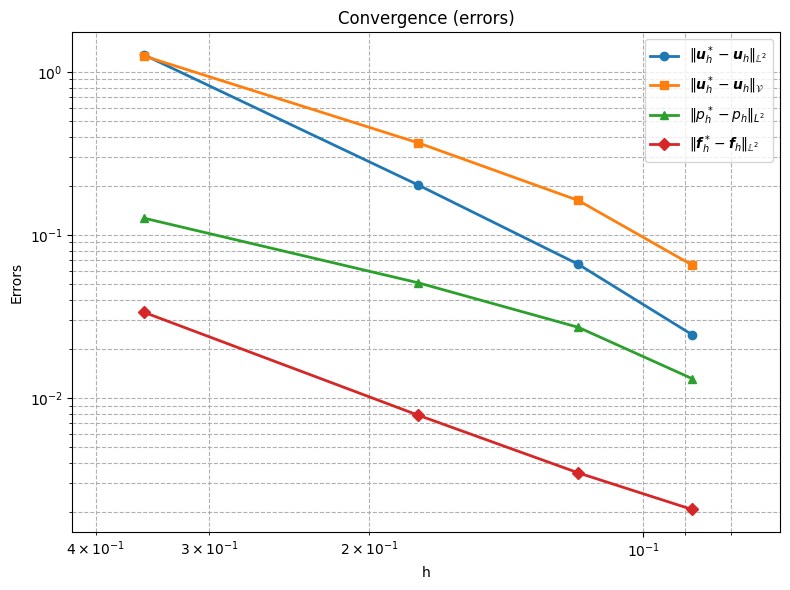}
		\caption{Errors vs. mesh size.}
		\label{fig:plot22}
	\end{subfigure}
	\caption{Cost history and plots of errors.}
	\label{fig:Ex2-errors}
\end{figure}

\begin{table}[htbp]
	\centering
	\begin{tabular}{ccccc}
		\hline
		$h$ & $\|\bu_h^*-\bu_h\|_{\L^2}$  & $\|\bu_h^*-\bu_h\|_{\mathcal{V}}$ & $\|p_h^*-p_h\|_{L^2}$ & $\|\f_h^*-\f_h\|_{\L^2}$  \\
		\hline
		$3.5355\!\times\!10^{-1}$ & $1.266\! \times\!10^{0}$   & $1.246\!\times\!10^{0}$   & $1.265\!\times\!10^{-1}$   & $3.347\!\times\!10^{-2}$    \\
		$1.7678\!\times\!10^{-1}$ & $2.023\!\times\!10^{-1}$  & $3.668\!\times\!10^{-1}$  & $5.089\!\times\!10^{-2}$  & $7.855\!\times\!10^{-3}$  \\
		$1.1785\!\times\!10^{-1}$ & $6.608\!\times\!10^{-2}$  & $1.627\!\times\!10^{-1}$  & $2.715\!\times\!10^{-2}$  & $3.480\!\times\!10^{-3}$  \\
		$8.8388\!\times\!10^{-2}$ & $2.451\!\times\!10^{-2}$  & $6.549\!\times\!10^{-2}$  & $1.318\!\times\!10^{-2}$  & $2.078\!\times\!10^{-3}$  \\
		\hline
	\end{tabular}
	\caption{Errors for velocity ($\bu_h$), pressure ($p_h$), and control ($\f_h$).} \label{tab:ex2-errors-pred}
\end{table}

\subsection{Example 3}
Here, we choose the desired vector $\bu_d$ and pressure $p_d$ as 
\begin{align*}
\boldsymbol{u}_d(x,y) =
\begin{pmatrix}
	-\cos(2\pi x)\,\sin(2\pi y) + \sin(2\pi y) \\[6pt]
	\sin(2\pi x)\,\cos(2\pi y) - \sin(2\pi x)
\end{pmatrix},
\
p_d(x,y) = 2\pi \big( \cos(2\pi y) - \cos(2\pi x) \big).
\end{align*}

In this case, we choose initial control as $\f_h^0 = (\sin(2\pi y), -\sin(2\pi x)).$ The stopping parameters are fixed as 
$\varepsilon_{\mathrm{hvi}} = 10^{-5}$ for the Uzawa-Newton inner solver 
and $\varepsilon_{\mathrm{opt}} = 10^{-5}$ for the outer iteration. The 
step-size sequence is again chosen as $\tau_k = 10^{-2}$. In the finest mesh, the optimal 
control $\f_h^*$, together with the associated optimal velocity $\bu_h^*$ 
and pressure $p_h^*$, is shown in Figure~\ref{fig:Ex2-plot-oVCP}. The 
history of the cost functional on the finest mesh is plotted in 
Figure~\ref{fig:Ex2-CostHist}. The relative errors, computed by using the 
solution on the finest mesh as the reference, are summarized in 
Table~\ref{tab:Ex3convergence_rates} and illustrated in Figure~\ref{fig:Ex2-errors}. 
Once again, the errors in the velocity, pressure, and control decrease 
consistently, thereby validating the convergence of the proposed method.

\begin{table}[h!]
	\centering
	\begin{tabular}{ccccc}
		\hline
		$h$ & $\|\boldsymbol{u}_h^*-\boldsymbol{u}_h\|_{\mathbb{L}^2}$ &
		$\|\boldsymbol{u}_h^*-\boldsymbol{u}_h\|_{\mathcal{V}}$ &
		$\|p_h^*-p_h\|_{L^2}$ &
		$\|\boldsymbol{f}_h^*-\boldsymbol{f}_h\|_{\mathbb{L}^2}$ \\
		\hline
		$3.5355{\times}10^{-1}$ & $6.488{\times}10^{-4}$ & $4.177{\times}10^{-3}$ & $2.425{\times}10^{-3}$ & $5.758{\times}10^{-3}$ \\
		$1.7678{\times}10^{-1}$ & $9.298{\times}10^{-5}$ & $8.942{\times}10^{-4}$ & $9.208{\times}10^{-4}$ & $1.373{\times}10^{-3}$ \\
		$1.1785{\times}10^{-1}$ & $2.853{\times}10^{-5}$ & $3.622{\times}10^{-4}$ & $5.015{\times}10^{-4}$ & $6.815{\times}10^{-4}$ \\
		$8.8388{\times}10^{-2}$ & $9.232{\times}10^{-6}$ & $1.380{\times}10^{-4}$ & $2.335{\times}10^{-4}$ & $4.692{\times}10^{-4}$ \\
		\hline
	\end{tabular}
	\caption{Errors for velocity in $\mathbb{L}^2$, velocity in $\mathcal{V}$, pressure in $L^2$, and control in $\mathbb{L}^2$.} \label{tab:Ex3convergence_rates}
\end{table}

\begin{figure}[ht]
	\centering
	\includegraphics[width=\textwidth]{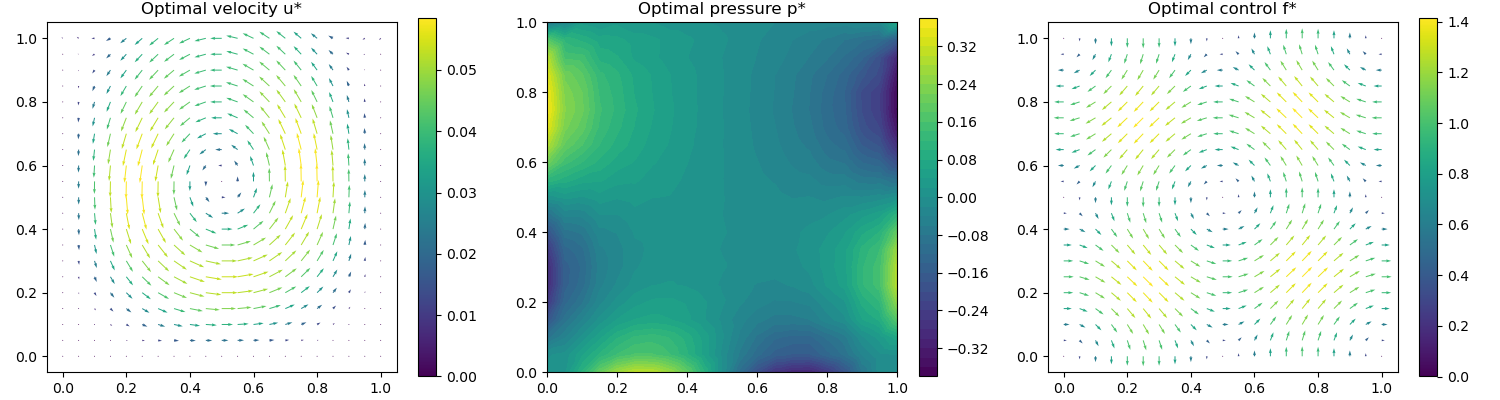}
	\caption{Plot of optimal velocity $\bu^*$, optimal pressure $p^*$, and optimal control $\f^*$ in mesh size $25\times 25$.}
	\label{fig:Ex3-plot-oVCP}
\end{figure}

\begin{figure}[ht]
	\centering
	\begin{subfigure}{0.48\textwidth}
		\centering
		\includegraphics[width=\linewidth]{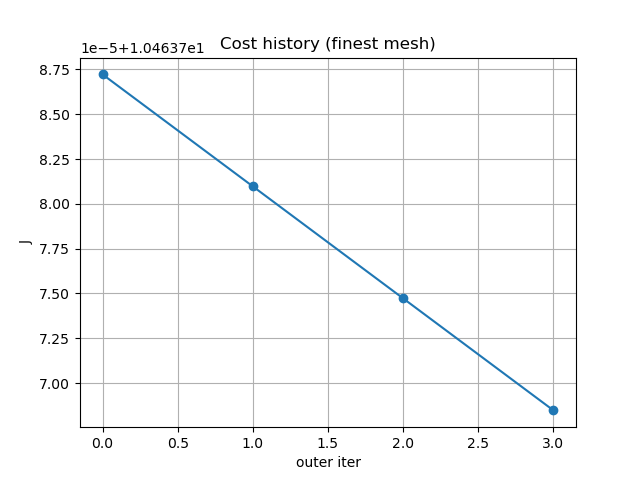}
		\caption{Cost history.}
		\label{fig:Ex3-CostHist}
	\end{subfigure}
	\hfill
	\begin{subfigure}{0.48\textwidth}
		\centering
		\includegraphics[width=\linewidth]{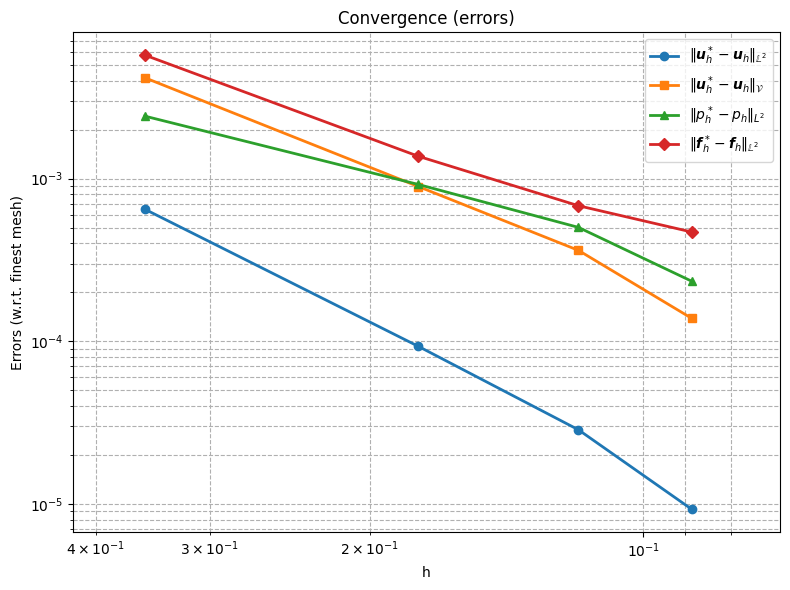}
		\caption{Pressure error vs. mesh size.}
		\label{fig:Ex3}
	\end{subfigure}
	\caption{Cost history and plots of errors.}
	\label{fig:Ex3-errors}
\end{figure}

\medskip
\noindent
\textbf{Acknowledgments:}  Support for M. T. Mohan's research received from the National Board of Higher Mathematics (NBHM), Department of Atomic Energy, Government of India (Project No. 02011/13/2025/NBHM(R.P)/R\&D II/1137). W. Akram is supported by NBHM (National Board of Higher Mathematics, Department of Atomic Energy) postdoctoral fellowship, No. 0204/16(1)(2)/2024/R\&D-II/10823.

	\end{document}